\documentclass[a4paper,reqno,11pt]{article}

\usepackage{amsmath, amsfonts, amssymb, amsthm, amscd}
\usepackage{graphicx}
\usepackage{psfrag}
\usepackage{perpage}
\usepackage{url}
\usepackage{color}
\usepackage{lmodern}
\usepackage{dsfont}
\usepackage{epsf}
\usepackage{graphics}
\usepackage{epsfig}
\usepackage{epstopdf} 
\usepackage{enumerate}

\usepackage[utf8]{inputenc}
\usepackage[T1]{fontenc}
\usepackage{microtype}

\usepackage[a4paper,scale={0.72,0.74},marginratio={1:1},footskip=7mm,headsep=10mm]{geometry}

\usepackage{hyperref}

\usepackage{verbatim} 

\numberwithin{equation}{section}

\newtheorem{theorem}{Theorem}[section]
\newtheorem{lemma}[theorem]{Lemma}

\newtheorem{definition}[theorem]{Definition}

\newcommand{\bbE}{{\ensuremath{\mathbb E}} }

\newcommand{\bbN}{{\ensuremath{\mathbb N}} }

\newcommand{\bbP}{{\ensuremath{\mathbb P}} }

\newcommand{\bbR}{{\ensuremath{\mathbb R}} }

\newcommand{\bbZ}{{\ensuremath{\mathbb Z}} }

\newcommand{\cA}{{\ensuremath{\mathcal A}} }
\newcommand{\cB}{{\ensuremath{\mathcal B}} }
\newcommand{\cC}{{\ensuremath{\mathcal C}} }

\newcommand{\cE}{{\ensuremath{\mathcal E}} }
\newcommand{\cF}{{\ensuremath{\mathcal F}} }
\newcommand{\cG}{{\ensuremath{\mathcal G}} }
\newcommand{\cH}{{\ensuremath{\mathcal H}} }
\newcommand{\cI}{{\ensuremath{\mathcal I}} }

\newcommand{\cN}{{\ensuremath{\mathcal N}} }
\newcommand{\cO}{{\ensuremath{\mathcal O}} }

\newcommand{\cV}{{\ensuremath{\mathcal V}} }

\newcommand{\ga}{\alpha}

\newcommand{\gep}{\varepsilon}       
\newcommand{\gz}{\zeta}

\newcommand{\gl}{\lambda}
\newcommand{\gL}{\Lambda}

\newcommand{\fC}{{\ensuremath{\mathfrak C}} }

\renewcommand{\tilde}{\widetilde}          
\DeclareMathSymbol{\leqslant}{\mathalpha}{AMSa}{"36} 
\DeclareMathSymbol{\geqslant}{\mathalpha}{AMSa}{"3E} 
\DeclareMathSymbol{\eset}{\mathalpha}{AMSb}{"3F}     
\newcommand{\dd}{\text{\rm d}}             

\newcommand{\inftwo}[2]{\inf_{\substack{#1 \\ #2}}} 
\newcommand{\sumtwo}[2]{\sum_{\substack{#1 \\ #2}}} 

\newcommand{\one}{{\mathchoice {1\mskip-4mu\mathrm l}
         {1\mskip-4mu\mathrm l}
         {1\mskip-4.5mu\mathrm l}
         {1\mskip-5mu\mathrm l}}}

\newcommand{\R}{\mathbb{R}}

\newcommand{\Z}{\mathbb{Z}}
\newcommand{\N}{\mathbb{N}}

\def\bs{\boldsymbol}

\newcommand{\PEfont}{\mathrm}

\DeclareMathOperator{\cov}{\ensuremath{\PEfont Cov}}
\def\p{\ensuremath{\PEfont P}}
\def\e{\ensuremath{\PEfont E}}
\newcommand{\E}{\e}
\renewcommand{\P}{\p}

\newcommand\bC{\ensuremath{\bs{\mathrm{C}}}}

\renewcommand{\epsilon}{\varepsilon}
\renewcommand{\rho}{\varrho}
\renewcommand{\phi}{\varphi}

\newcommand{\sC}{{\ensuremath{\mathfrak C}} }

\newenvironment{myenumerate}{%
\renewcommand{\theenumi}{\arabic{enumi}}%
\renewcommand{\labelenumi}{{\rm(\theenumi)}}%
\begin{list}{\labelenumi}
	{%
	\setlength{\itemsep}{0.4em}%
	\setlength{\topsep}{0.5em}%
	\setlength\leftmargin{2.45em}%
	\setlength\labelwidth{2.05em}%
	\setlength{\labelsep}{0.4em}%
	\usecounter{enumi}%
	}%
	}%
{\end{list}
}

{\end{myenumerate}}

\newenvironment{myitemize}{%
\begin{list}{$\bullet$}%
 	{%
	\setlength{\itemsep}{0.4em}%
	\setlength{\topsep}{0.5em}%
	\setlength\leftmargin{2.45em}%
	\setlength\labelwidth{2.05em}%
	\setlength{\labelsep}{0.4em}%
	}%
	}%
{\end{list}}

\renewenvironment{itemize}{
\begin{myitemize}}%
{\end{myitemize}}

\MakePerPage[2]{footnote} 

\def\dd{\mathrm{d}}

\newcommand{\dist}{\text{\rm dist}} 
\newcommand{\Leb}{\text{\rm Leb}}

\newcommand{\inte}{\text{\rm int}} 
\newcommand{\ext}{\text{\rm ext}}

\newcommand{\cp}{\mathrm{cap}}
\newcommand{\vol}{\mathrm{vol}}

\newcommand{\ws}[2]{W^{ #1}_{ #2}}

\definecolor{light-gray}{gray}{0.5}

\newcommand{\be}{\begin{equation}}
\newcommand{\ee}{\end{equation}}
\newcommand{\ba}{\begin{aligned}}
\newcommand{\ea}{\end{aligned}}
\newcommand{\ind}[1]{\one_{\{ #1\}}}

\newcommand{\good}{\text{good}}
\newcommand{\esssup}{\mathrm{esssup}}

\begin{document}


\title{Asymptotics of the critical time in Wiener sausage percolation with a small radius}

\author{
Dirk Erhard
\footnotemark[1]
\\
Julien Poisat
\footnotemark[3]
}

\footnotetext[1]{
Mathematics Institute, University of Warwick,
Coventry, CV4 7AL, UK,\\
{\sl D.Erhard@warwick.ac.uk}
}

\footnotetext[3]{
CEREMADE, Université Paris-Dauphine, PSL Research University, UMR 7534
Place du Maréchal de Lattre de Tassigny,
75775 PARIS CEDEX 16 - FRANCE\\
{\sl poisat@ceremade.dauphine.fr}
}

\date{\today}
\maketitle

\begin{abstract}
We consider a continuum percolation model on $\R^d$, where $d\geq 4$.
The occupied set is given by 
the union of independent Wiener sausages with radius $r$ running up to time $t$ and whose
initial points are distributed according to a homogeneous Poisson point process.
It was established in a previous work by Erhard, Mart\'{i}nez and Poisat~\cite{EMP13} that (1) if $r$ is small enough there is a non-trivial percolation transition
in $t$ occurring at a critical time $t_c(r)$ and (2) in the supercritical regime the unbounded cluster is unique. In this paper we investigate the asymptotic behaviour of the critical time when the radius $r$ converges to $0$. The latter does not seem to be deducible from simple scaling arguments. We prove that for $d\geq 4$, there is a positive constant $c$ such that
$c^{-1}\sqrt{\log(1/r)}\leq t_c(r)\leq c\sqrt{\log(1/r)}$ when $d=4$ and $c^{-1}r^{(4-d)/2}\leq t_c(r) \leq c\ r^{(4-d)/2}$ when $d\geq 5$, as $r$ converges to $0$. We derive along the way moment and large deviation estimates on the capacity of Wiener sausages, which may be of independent interest.

\medskip\noindent

{\it Acknowledgments.} The authors are grateful to A.\@ Sapozhnikov and G.\@ F.\@ Lawler for helpful suggestions. DE and JP were supported by ERC Advanced Grant 267356 VARIS. DE was supported by the ERC Consolidator Grant of Martin Hairer. DE and JP acknowledge the hospitality of Université Paris-Dauphine and of the University of Warwick, where part of this work was done. 
\end{abstract}

\newpage




\section{Introduction}
\label{S1}

{\it Notation.} For every $d\geq1$, we denote by $\Leb_d$ the Lebesgue measure on $\R^d$. The symbol
$||\cdot||$ stands for the Euclidean norm on $\R^d$ and the symbols $|\cdot|_1$ and $|\cdot|_\infty$ stand for the $\ell_1$ and $\ell_\infty$ norms on $\bbZ^d$, respectively.
The open ball with center $z$, radius $r$ and with respect to the Euclidean norm is denoted by $\cB(z,r)$, the closed ball by $\overline{\cB}(z,r)$, and $c_{\mathrm{vol}} = \Leb_d(\cB(0,1))$.
For $A\subseteq \R^d$ and $x\in\R^d$, we denote by $d(x,A)$ the Euclidean distance between $x$ and $A$, i.e.\ $d(x,a) = \inf_{y\in A}\{||x-y||\}$. The complement of a set $A$ is denoted by $A^c$ and its closure by $\bar A$ (the topology will depend on the context).
For two sets $A_1, A_2 \subseteq \R^d$, we denote by $A_1 \oplus A_2$ their Minkowski sum, defined by $\{x_1 + x_2,\ x_1\in A_1,\ x_2\in A_2\}$. 
For $a\in\R$, we denote by $\lceil a \rceil$ its upper integer part. The symbol $|\cdot|$ stands  for the cardinality of a set or the absolute value of a real number, depending on the context.
We denote by $\one$ the infinite column vector with all entries equal to one.
We denote by $G:\R^d\times\R^d\to [0,\infty)$ the Green function of the standard Brownian motion. Given $f$ and $g$ two positive functions we write $f\lesssim g$ if there is a constant $c\in (0,\infty)$ so that $f\leq c g$.\\

Throughout the paper the letter $c$ is used to denote a constant whose precise value is irrelevant (possibly depending on the dimension) and which may change from line to line.\\

\subsection{Introduction to the model}
\label{S1.1}
\par Let $\cE$ be a Poisson point process with intensity $\lambda\ \Leb_d$, where $\gl > 0$. Conditionally on $\cE$, we define a collection of independent Brownian motions $\{(B_t^x)_{t\geq 0},\, x\in \cE\}$ such that
for each $x\in\cE$, $B_0^x = x$ and $(B_t^x-x)_{t\geq0}$ is independent of $\cE$. We refer the reader to Section 1.4 in \cite{EMP13} for a rigorous construction. Let $\P$ and $\E$ be the probability measure and expectation of Brownian motion, respectively. We denote by $W_{[0,t]}^{x,r} = \bigcup_{0\leq s\leq t}\cB(B_s^x,r) = B^x_{[0,t]} \oplus \cB(0,r)$ the Wiener sausage with radius $r$, started at $x$ and running up to time $t$. When it is more convenient, we shall use $\P_x$ for a Brownian motion started at $x$, and we remove the superscript $x$ from $B$ or $W$. Also, we will use the symbol $\tilde\P$ to refer to an independent copy of a Brownian motion. If $A$ is an event, then $\E(\,\cdot\, ; A)$ stands for $\E(\, \cdot\, \one_A)$. Finally, we use the letter $\bbP$ for the law of the whole process that is formed by the Poisson points and the Brownian motions. \\

\par The object of interest is the {\it occupied} set defined by
\begin{equation}\label{def:occB}
\cO_{t,r} := \bigcup_{x\in\cE} W_{[0,t]}^{x,r},\qquad \cO_{t}:=\bigcup_{x\in \cE} B^x_{[0,t]}, \qquad t\geq 0,r> 0.
\end{equation}
The rigorous construction found in \cite{EMP13} yields ergodicity of $\cO_{t,r}$ with respect to shifts in space.
For $d\geq 4$, \u{C}ern\'{y}, Funken and Spodarev \cite{CFS08} used this model to describe the target detection area of a network of mobile sensors initially distributed at random and moving according to Brownian motions. In a similar spirit Kesidis, Konstantopoulos and Phoha \cite{KKP05} study the detection time of a particle that is placed at the origin. Note that at time $t=0$, the occupied set reduces to a collection of balls with randomly located centers: this goes under the name of Boolean percolation model and was first introduced by Gilbert \cite{G61} to study infinite communication networks. We refer to Meester and Roy \cite{MR96} for an introductory overview of this model.\\

\par Two points $x$ and $y$ of $\bbR^d$ are said to be {\it connected} in $\cO_{t,r}$ if and only if there exists a continuous function $\gamma: [0,1]\mapsto
\cO_{t,r}$ such that $\gamma(0)=x$ and $\gamma(1)=y$. A subset of $\cO_{t,r}$ is connected if and only if all of its points are pairwise connected, and a connected subset of $\cO_{t,r}$ is called a component. A component $\cC$ is bounded if there exists $R>0$ such that $\cC \subseteq \cB(0,R)$.
Otherwise, the component is said to be unbounded. A {\it cluster} is a connected component which is maximal for the inclusion. Denote by $\cC(x)$ the set of points in $\cE$ which are connected to $x$ through $\cO_{t,r}$.

\par A set is said to percolate if it contains an unbounded connected component. In \cite{EMP13} it was shown that $\cO_{t,r}$ undergoes a non-trivial percolation phase transition for all $d\geq 2$. More precisely it was shown that if $d\in\{2,3\}$, then for all $\lambda>0$ there exists $t_c(\gl)\in (0,\infty)$ such that for all $t<t_c(\gl)$ the set $\cO_{t}$ only contains bounded connected components, whereas for $t>t_c(\gl)$, the set $\cO_{t}$ percolates with a unique unbounded cluster. What happens at criticality is still unknown.
In essence the same result holds for $d\geq 4$. However, due to the fact that the paths of two independent Brownian motions do not intersect (except at a possibly common starting point), the set $\cO_{t,0}$ almost surely (a.s.) does not percolate for all $t\geq 0$. Therefore, the radius $r$ needs to be chosen positive. In this case, denote by $\lambda_c(r)$ the critical value such that the set $\cO_{0,r}$ a.s.\ percolates for all $\lambda > \lambda_c(r)$, and a.s. does not for $\lambda < \lambda_c(r)$, see Section 3.3 in Meester and Roy~\cite{MR96}. Theorem 1.3 in \cite{EMP13} states that when $r>0$ and $\lambda <\lambda_c(r)$, then there is a critical time $t_c(\lambda,r)\in (0,\infty)$ which separates a percolation regime ($t>t_c(\gl, r)$) from a non-percolation regime ($t<t_c(\gl,r)$). Equivalently, a phase transition occurs when $\lambda$ is fixed and the radius is chosen smaller than a critical radius $r_c(\gl)$. We choose the last formulation, which is more relevant for the rest of the paper.

\subsection{Main Result}
\label{S1.2}
In this paper we study the behaviour of the critical time as the radius converges to $0$ and the intensity is kept fixed to $\gl=1$. For this reason, we shall now write $t_c(r)$ instead of $t_c(1,r)$.
Let us mention that no simple scaling argument seems to immediately yield bounds on $t_c(r)$. Indeed, since for each $d$ there are three parameters ($\lambda$, $t$ and $r$), it is not possible to scale two parameters independently of the third one. We expect that $t_c(r)$ goes to $\infty$ as $r\to 0$, since $t_c(0)=\infty$. Note that this is not an immediate consequence of continuity since the event $\{\cO_t \mbox{ does not percolate}\}$ is not the increasing union of the events $\{\cO_{t,r} \mbox{ does not percolate}\}$ for $r>0$. The following theorem however confirms our intuition and determines at which speed the convergence takes place.

\begin{theorem}
\label{thm:5}
Let $d\geq 4$. There is a constant $c$ and an $r_0 \in (0,1)$ such that for all $r\leq r_0,$
\begin{equation}
\label{eq:thm}
\left\{
\begin{array}{ll}
c^{-1} \sqrt{\log(1/r)}\leq t_c(r)\leq c \sqrt{\log(1/r)}, &\mbox{if } d=4,\\
&\\
c^{-1} r^{(4-d)/2}\leq t_c(r) \leq c\ r^{(4-d)/2}, &\mbox{if } d\geq 5.
\end{array}\right.
\end{equation}
\end{theorem}

\subsection{Discussion}
\label{S1.3}

Items (1)--(3) below contain comments about the result. Items (4)--(5) are general comments about the model.\\

\noindent \textrm{\bf (1)} For completeness, we state that $r\mapsto t_c(r)$ stays bounded as $r\to 0$ when $d\in\{2,3\}$, since, by monotonicity, $\limsup_{r\to 0} t_c(r) \leq t_c(0) <\infty$. This follows from \cite[Theorem 2]{EMP13}. Continuity at $r=0$ is not immediate, but we expect that this follows from a finite-box criterion of percolation.  Theorem \ref{thm:5} shows in particular that when $d\geq 4$ the critical time is continuous at $r=0$, since $t_c(0) = \infty$.
\medskip

\noindent \textrm{\bf (2)} 
One motivation to study the small radius asymptotics of the critical time is to gain a better understanding of the percolation mechanisms when $d\geq 4$. Indeed, when $d\in\{2,3\}$ percolation can occur because two independent Brownian motions that start close to each other eventually intersect, see \cite[Lemma 5.1]{EMP13}. This argument however breaks down when $d\geq 4$. The proof of Theorem \ref{thm:5} gives some insight on how percolation occurs in that case. 
\medskip

\noindent \textrm{\bf (3)}
The proof of our result makes use of moment and large deviation estimates on the capacity of a Wiener sausage, that we derive in Section \ref{S5}. When $d=4$, these are more subtle and therefore require a more careful analysis than in the high dimensional case $d\geq 5$. This is due to the logarithmic correction in the increase of the mutual intersection local time in four dimensions. For similar moment estimates in the case of simple random walk, we refer to Rath and Sapozhnikov \cite{RS12} ($d\geq 5$) and Chang and Sapozhnikov~\cite[Equation (4)]{CS14} ($d=4$). Let us mention that while preparing this manuscript we were getting aware of a work in progress by van den Berg, Bolthausen and den Hollander~\cite{vdBBdH} who developed simultaneously to us capacity estimates that are similar in spirit. 
\medskip

%

\noindent \textrm{\bf (4)} Random interlacement is a Poisson point process on infinite random walk paths obtained when looking at the trace of a simple random walk on the torus $(\bbZ/N\bbZ)^d$ started from the uniform distribution, running up to time $u N^d$ and letting $N\nearrow\infty$, see Sznitman~\cite{S10}. We expect that, as $t\nearrow\infty$ , $\gl\searrow 0$ and $\gl t$ stays constant, while $r$ is fixed, our model shares features with a continuous version of random interlacements, see Sznitman~\cite{S13}.  Indeed, in the regime described above, the number of Brownian trajectories entering a set $A$ is a Poisson random variable with intensity proportional to $\gl t\ \cp(A)$, which is a key feature of random interlacements. Moreover, the product of $\gl t$ serves as an intensity parameter.  This limiting regime exhibits long-range dependence, in the sense that if $A_1$ and $A_2$ are two bounded sets, then 
\begin{equation}
\cov(\one_{\{A_1\cap\cO_t \neq \emptyset\}},\one_{\{A_2\cap\cO_t \neq \emptyset\}}) \sim c\ \dist(A_1,A_2)^{2-d},
\end{equation}
as $\dist(A_1,A_2)\nearrow\infty$, $t\nearrow\infty$ and $\gl t$ stays constant. Indeed, the left-hand side becomes asymptotically equivalent to the difference between $\cp(A_1\cup A_2)$ and $\cp(A_1) + \cp(A_2)$, which has the desired order.

\medskip

\noindent \textrm{\bf (5)} Peres, Sinclair, Sousi and Stauffer \cite{PSSS13, PSS13} also study a system of points randomly distributed in space and moving according to Brownian motions. However, instead of only looking at $\cO_{t,r}$, they also look at $\Sigma_{t,r} = \cup_{x\in \cE}\cB(B_t^x,r)$ at \emph{each fixed time }t. Nevertheless, in contrast to our setting, they choose $r$ large enough such that $\Sigma_{t,r}$ contains an unbounded cluster for all $t\geq 0$. In these papers the focus is on three aspects:
\begin{itemize}
\item[(i)] detection (the first time that a target point is contained in $\Sigma_{t,r}$);
\item[(ii)] coverage (the first time that all points inside a finite box are contained in $\cO_{t,r}$);
\item[(iii)] detection by the unbounded cluster (the time it takes until a target point belongs to the unbounded cluster of $\Sigma_{t,r}$).
\end{itemize}  

\subsection{Open questions}
\label{S1.3b}

\noindent \textrm{\bf (1)} Do the upper and lower bounds in Theorem \ref{thm:5} match? More precisely, is there a $c_*\in (0,\infty)$ such that
\begin{equation}
\label{eq:f}
\lim_{r\to 0}t_c(r)/f(r) = c_*,\quad \mbox{with}\quad f(r)=\left\{
\begin{array}{ll}
r^{(4-d)/2}, &d\geq 5,\\
\sqrt{\log(1/r)}, &d=4?
\end{array}\right.
\end{equation}
\medskip

\noindent \textrm{\bf (2)} Is there a way to define a limiting random subset of $\bbR^d$ as we set time to $t(r)= c f(r)$ (see \eqref{eq:f}) and intensity to $\gl = 1$ in our model, and let $r\searrow 0$? Would this limiting object have a percolation phase transition in $c$ and if so, would the critical value of $c$ coincide with the constant $c_*$  in \eqref{eq:f}? Note that one should beforehand perform a change of parameters. Indeed, for any couple of points $x$ and $y$  in $\bbR^d$, the intersection of $\ws{x,r}{[0,\infty)}$ and $\ws{y,r}{[0,\infty)}$ becomes eventually empty as $d\geq4$ and $r\searrow 0$, meaning that percolation occurs out of arbitrarily large windows and is thus not visible in the limit. To fix this issue, one may set time $t=1$ such that intersections of Wiener sausages occur in a space window that remains bounded, and in order to be consistent with the previous scaling, let the intensity parameter be $\gl(t) = t^{d/2}$ and the radius parameter be $r(t) = c t^{\frac{d}{2(4-d)}}$ if $d\geq 5$ and $r(t) = t^{-1/2}e^{-t^2/c}$ if $d=4$, with a different $c$. 

\subsection{Outline}
\label{S1.4}

In Section 2, we recall facts about the Green function and the Newtonian capacity. Section 3 contains the proof of the lower bound, which is guided by the following idea: suppose that the origin is contained in the occupied set, then perform a tree-like exploration of the cluster containing the origin and dominate it by a sub-critical Galton-Watson branching process. Extinction of the Galton-Watson process implies non-percolation of the cluster. Section 4 contains the proof of the upper bound, which consists in the following coarse-graining procedure: (i) we split space in an infinite collection of balls all having a radius of the order $\sqrt{t}$, (ii) each ball is shown to contain with high probability the starting point of a Wiener sausage whose Newtonian capacity is large enough, and (iii) provided $t$ is large enough, these Wiener sausages form an unbounded connected component. Finally, Section 5 contains the proof of several capacity estimates that we use along Sections 3 and 4.

\section{Preliminaries on Green function and capacity}
\label{S2}
In this section we introduce the notion of capacity. We refer the reader to M\"orters and Peres \cite{MP10} as well as Port and Stone \cite{PS78} for more detailed surveys on this subject.
Let $d\geq 3$ and denote by $\Gamma$ the Gamma function. The Green function associated with Brownian motion on $\R^d$ is defined
as 
\begin{equation}
\label{eq:green}
G(x,y) = \frac{\Gamma(d/2-1)}{2\pi^{d/2}||x-y||^{d-2}}, \quad x,y\in\R^d.
\end{equation}

\begin{definition}
Let $A\subseteq \R^d$ be a Borel set. The energy of a finite Borel measure $\nu$ on $A$
is defined as
\begin{equation}
\label{eq:energy}
\cI(\nu) = \int_A\int_A G(x,y)\nu(\dd x)\nu(\dd y)
\end{equation}
and the Newtonian capacity of $A$ is defined as
\begin{equation}
\label{eq:cap}
\mathrm{cap}(A) = [\inf_{\nu} \cI(\nu)]^{-1},
\end{equation}
where the infimum is over all probability measures on $A$.
\end{definition}
\noindent
Let $A,A'$ be bounded Borel sets. The function $A\mapsto \mathrm{cap}(A)$ is non-decreasing in $A$, satisfies
the scaling relation
\begin{equation}
\label{eq:scaling}
\cp(aA) = a^{d-2} \cp(A), \quad a>0,
\end{equation}
and is a submodular set function:
\begin{equation}
\label{eq:unionbound}
\cp(A\cup A') + \cp(A\cap A') 
\leq \cp(A) + \cp(A').
\end{equation} 
Given a bounded set $A\subseteq \R^d$, let $\tau_A$ be the last exit time of $A$ (with the convention that $\tau_A = 0$ if the Brownian motion does not visit the set $A$). There exists a finite measure $e_A$ on $A$, the equilibrium measure of $A$, such that for any Borel set $\Lambda\subseteq A$ and every $x\in\R^d$ (see Chapter 3, Theorem 2.1 in \cite{PS78}),
\begin{equation}
\label{eq:hittingprob}
\P_x(B_{\tau_A}\in \Lambda, \tau_A>0) 
= \int_\Lambda G(x,y) e_A(\dd y)
\end{equation}
and such that
\begin{equation}
\label{eq:cap.e}
\cp(A) = e_A(A). 
\end{equation}
It moreover has an interpretation in terms of hitting probabilities:
\begin{equation}
\label{eq:hittingint}
\lim_{||x||\to\infty}||x||^{d-2}\P(B^{x}_{[0,\infty)}\cap A \neq \emptyset)
= \frac{\cp (A)}{\kappa_d}, \quad A\subseteq \R^d \mbox{ bounded Borel set},
\end{equation}
where $\kappa_d= 2\pi^{d/2}/\Gamma(d/2-1)$ is the capacity of the unit ball
(see Chapter 3, Theorem 1.10 in \cite{PS78}).
Finally, the Poincaré-Faber-Szeg\"o inequality~\cite{PS51} states that for any bounded, open set $A\subseteq \R^d$
\begin{equation}
\label{eq:PFSstatement}
\Leb_d(A)\lesssim \cp(A)^{d/(d-2)}.
\end{equation}
Here, the proportionality constants only depends on the dimension.
\section{Proof of the lower bound}
\label{S3}
In this section we prove the lower bound of Theorem~\ref{thm:5}. The proof for the case $d\geq 5$ is given in Section~\ref{S3.1} and the proof for the case $d=4$ is given in Section~\ref{S3.2}.
Throughout this section we use the abbreviations
\begin{equation}\label{eq:abb.int}
x\sim y \quad \Longleftrightarrow \quad W_{[0,t]}^{x,r} \cap W_{[0,t]}^{y,r} \neq \emptyset,\qquad \cN(x) = \{y\in\cE\setminus\{x\} \colon x\sim y\},\qquad x,y\in\cE
\end{equation}
and for a set $A\subseteq \R^d$ we write
\begin{equation}
\label{eq:outradius}
M(A)= \sup_{x\in A}||x||
\end{equation}
for the outradius of $A$.
Let us stress that $\cN(x)$ also depends on $t$, so that one may also use the notation $\cN_t(x)$ instead. For ease of readability however we abstain from using $t$ in the notation.

\subsection{Case $d\geq 5$}
\label{S3.1}

\par We use a technique that has been used in the context of Boolean percolation, which consists of exploring the cluster containing the origin and comparing it to a (multitype) Galton-Watson branching process, see for instance Meester and Roy~\cite[Section 3.3]{MR96}. For simplicity, we assume that there is a Poisson point at the origin, which is justified in the proof of Lemma \ref{lemmaA} below. For that purpose we introduce $\bbP^0$ the law of our process after addition of a Brownian motion at the origin. The Wiener sausages intersecting the Wiener sausage starting at the origin are called first generation sausages, all other sausages intersecting the first generation sausages constitute the second generation sausages, and so on. 
This leads to the following decomposition of $\cC(0)$:
\be
\cE_0 = \{0\}, \qquad \cE_{n+1} =
\begin{cases}
\bigcup_{y\in \cE_n} \cN(y) \setminus \bigcup_{k=0}^n \cE_k & \mbox{ if } \cE_n \neq \emptyset\\
\emptyset &\mbox{ if } \cE_n = \emptyset
\end{cases},\quad n\in\N_0.
\ee
Here $\cE_n$ is interpreted as the set of elements in $\cC(0)$ at generation $n$.
The idea is to dominate the process $\{|\cE_n|\}_{n\in\N_0}$ by a branching process which eventually becomes extinct, thus proving that $\cC(0)$ contains finitely many Poisson points, which in turn proves non-percolation.
If this branching process would be close (in some reasonable sense) to a Galton Watson process, then it would be enough to control the mean number of offsprings of the Wiener sausage started at the origin. However, the Wiener sausages of the first generation are not distributed as Wiener sausages but as Wiener sausages conditioned to intersect $W^{0,r}_{[0,t]}$. These are subject to a size biasing effect meaning that their capacities have a bias towards larger values, compared to the unconditioned Wiener sausage. To overcome this difficulty we employ a multitype branching argument. More precisely, we partition the set of Poisson points according to the capacities of their associated Wiener sausages:
\begin{equation}\label{def:classes}
\sC_j = \Big\{x\in \cE \colon \cp(W_{[0,t]}^{x,3r}) \in [j,j+1)tr^{d-4}\Big\}, \qquad j\in\bbN_0.
\end{equation}
The term $tr^{d-4}$ above is due to the fact that $\E[\cp(W^{0,r}_{[0,t]})]$ is bounded from above and from below by a constant times $tr^{d-4}$, which can be deduced from the arguments used in Sections \ref{S5.2}--\ref{S5.3}. The reason to consider Wiener sausages with radius $3r$ instead of $r$ is of technical nature and does not hide anything deep.\\

We now introduce the auxiliary multitype branching process, see Athreya and Ney~\cite{AN72} for the necessary theory (in the case of a finite number of types). First, define
\begin{equation}
\label{eq:defnij}
N(i,j)= \esssup\, \bbE^0 \Big[ \big| \sC_j \cap \cN(0)  \big|\ \Big|\  B^0 \Big]\one\{0 \in \sC_i\}, \qquad i,j\in\bbN_0,
\end{equation}
which will be a parameter of the offspring distribution. Note that the supremum is taken over the realisations of $B^0$. Let $\gz$ be a $\N_0$-valued random variable with $\P(\gz = j) = \bbP^0(0\in\sC_j)$ for $j\in\N_0$ and independently from that, $\{\gz^{(i,j)}_{k,\ell}\}_{i,j,k,\ell}$ be independent Poisson random variables with parameter $N(i,j)$. Let $(Z_n^{(j)})_{j,\in\N_0}$, where $n$ stands for the generation number and $j$ the type, be defined by
\be
\label{eq:defZ0}
Z_0^{(j)} = 
\begin{cases}
1 & \mbox{ if } j=\gz\\
0 & \mbox{ else }
\end{cases}, \quad j\in\N_0,
\ee
and conditionally on $(Z_k^{(i)})_{0\leq k \leq n, i\in\N_0}$,
\be
\label{eq:defZ1}
Z_{n+1}^{(j)} = \sum_{i\in\N_0} \one\{Z_n^{(i)} \geq 1\} \sum_{\ell=1}^{Z_n^{(i)}} \gz^{(i,j)}_{n+1,\ell}.
\ee
In particular, conditionally on 
$(Z_k^{(i)})_{0\leq k \leq n, i\in\N_0}$, the random variables $Z_{n+1}^{(j)}$ with $j\in\N_0$ are Poisson distributed with parameter $\sum_{i\in\N_0}Z_n^{(i)} N(i,j)$. Set $Z_n = \sum_{i\in\N_0} Z_n^{(i)}$ for $n\in\N_0$ and note that if $Z_n=0$, then $Z_{n+m} = 0$ for all $m\in N_0$. We finally define the extinction time $\tau_\ext = \inf\{n\geq 1\colon Z_n = 0\}$.\\

The rest of the section is organized as follows: first we justify with Lemma \ref{lemmaA} why we can add a Poisson point at the origin (even though this can be considered standard, we have not found a rigorous argument in the literature). Lemma \ref{lemmaB} makes the link between $\cC(0)$ and the auxiliary multitype branching process. Then, in Lemma \ref{lemmaC}, we give a sufficient criterion on the kernel $N(i,j)$ defined in \eqref{eq:defnij} for the multitype branching process to become extinct. Lemma~\ref{lem:UB.Ntr} provides an upper bound on the $N(i,j)$'s. Finally, we combine all four lemmas to prove the lower bound in Theorem \ref{thm:5}.

\begin{lemma}
\label{lemmaA}
Let $t, r>0$ be such that $\cC(0)$ is a.s. finite under $\bbP^0$. Then, $\cO_{t,r}$ does not percolate, $\bbP$-a.s.
\end{lemma}

\begin{proof}
Since $\cE$ is a Poisson point process, $\bbP^0$ coincides with the Palm version of $\bbP$, see Proposition 13.1.VII in Daley and Vere-Jones~\cite{DVJ88}. By definition of the Palm measure, for all bounded Borel sets $A\subseteq\R^d$
\begin{equation}
\bbP^0(|\cC(0)|<\infty) = \frac{1}{\Leb_d(A)} \bbE\Bigg\{\sum_{x\in \cE \cap A} {\bf 1}\{|\cC(x)|<\infty\}\Bigg\}.
\end{equation}
Therefore, if $\bbP^0(|\cC(0)|<\infty) = 1$ then by choosing a sequence of Borel sets $(A_n)_{n\in \N}$ increasing to $\bbR^d$, we get that $\bbP$-a.s.\ all clusters are finite, which proves non-percolation.
\end{proof}

\begin{lemma}
\label{lemmaB}
If $\tau_\ext$ is a.s.\ finite then $\cC(0)$ is $\bbP^0$-a.s.\ finite.
\end{lemma}

\begin{proof}
Let us define for $j\in\N_0$ and $n\in\bbN_0$, $\cE_n^{(j)} = \cE_n \cap \sC_j$, $Y_n^{(j)} = |\cE_n^{(j)}|$ and $Y_n = \sum_{j\in\N_0} Y_n^{(j)} = |\cE_n|$. Note that $\cC(0)$ is finite if and only if there exists a (random) $n_0$ such that $Y_n = 0$ for $n\geq n_0$, or equivalently, $\sum_{n\in \N_0} Y_n$ converges. The idea is to dominate the  process $(Y_n)_{n\in\N_0}$ by the multitype branching process $(Z_n)_{n\in\N_0}$ defined in \eqref{eq:defZ0} and \eqref{eq:defZ1}. Let us first explain how this domination works for $n=1$. Define
\be
\widetilde{\cE}_n= \bigg\{x\in\cE:\, x\notin \bigcup_{m\leq n} \cE_m\bigg\},\qquad n\in\N_0.
\ee
Conditionally on $B^0$,  $\{\cE_1^{(j)}\}_{j\in\N_0}$ and $\widetilde{\cE}_1$  are independent Poisson point processes on $\bbR^d$ with respective intensity measures $\{\bbP^0(x\in\cN(0)\cap \sC_j | B^0)\dd x \}_{j\in\N_0}$ and $\bbP^0(x\notin\cN(0) | B^0)\dd x$. Recall \eqref{eq:defnij}. Conditionally on $B^0$, and on the event $\{0\in\sC_i\}$, $Y_1^{(j)}$ is therefore a Poisson random variable with a parameter smaller than $N(i,j)$, hence $Y_1^{(j)}$ is stochastically dominated by $Z_1^{(j)}$, as defined in \eqref{eq:defZ0}--\eqref{eq:defZ1}, for all $j\in\N_0$. Consequently, $Y_1$ is dominated by $Z_1$.

We now explain the iteration procedure. Let $\cG_n$ be the $\sigma$-field generated by $(\cE_k)_{k\leq n}$ and $(B^x,\ x\in \cup_{k\leq n} \cE_k)$ and let the properties ($\cH_1^{n}$) and ($\cH_2^{n}$) be defined by:\\
\noindent --- {\bf ($\cH_1^n$)} Conditionally on $\cG_{n-1}$, on the event $\{\cE_{n-1} \neq \emptyset\}$, the random sets $\{\cE_n^{(j)}\}_{j\in\N_0}$ and $\widetilde{\cE}_n$ are independent Poisson point processes with respective intensity measures $\{\bbP^0(x\in \cup_{y\in\cE_{n-1}} \cN(y) \cap \sC_j  \setminus \cup_{i=0}^{n-1} \cE_i  | \cG_{n-1})\dd x \}_{j\in\N_0}$ and $\bbP^0(x\notin\cup_{y\in\cE_{n-1}} \cN(y)  \cup_{i=0}^{n-1} \cE_i | \cG_{n-1})\dd x$;\\
\noindent --- {\bf ($\cH_2^n$)} Conditionally on $\cG_{n-1}$, on the event $\{\cE_{n-1} \neq \emptyset\}$, $Y_n^{(j)}$ is stochastically dominated by $Z_n^{(j)}$ for all $j\in\N_0$.\\
\noindent We now suppose that {\bf ($\cH_1^n$)} and {\bf ($\cH_2^n$)} are true and sketch how to conclude that {\bf ($\cH_1^{n+1}$)} and {\bf ($\cH_2^{n+1}$)} are also true, see Procaccia~and~Tykesson~\cite[Lemmas 7.1--7.2]{PT11} for details in a similar context. The sets $\{\cE_{n+1}^{(j)}\}_{j\in\N_0}$ and $\widetilde{\cE}_{n+1}$ are clearly disjoint and contained in $\widetilde{\cE}_n$, which, by induction hypothesis, is independent from $\cE_n$. Therefore, conditionnally on $\cE_n$, $\{\cE_{n+1}^{(j)}\}_{j\in\N_0}$ and $\widetilde{\cE}_{n+1}$ are independent Poisson processes with the desired intensity measures, which proves {\bf ($\cH_1^{n+1}$)}. We now turn to the proof of ($\cH_2^{n+1}$). Let $j\in \N_0$. By removing the restriction $x\notin \cup_{i=0}^n \cE_i$ and using the decomposition $\cE_n = \cup_{i=0}^\infty\cE_n^{(i)}$, we get the upper bound
\be
\bbP^0(x\in \cup_{y\in\cE_n} \cN(y) \cap \sC_j  \setminus  \cup_{i=0}^n \cE_i | \cG_n, x\in\cE) \leq \sum_{i\in \N_0} \bbP^0(x\in \cup_{y\in\cE^{(i)}_n} \cN(y) \cap \sC_j | \cG_n, x\in\cE).
\ee
By using translation invariance, conditionally on $\cG_n$, the parameter of $Y_{n+1}^{(j)}$ is bounded from above by $\sum_{i\in \N_0} Y_n^{(i)} N(i,j)$. Using {\bf ($\cH_2^n$)}, \eqref{eq:defZ1} and the remark following it, we see that $Y_{n+1}^{(j)}$ is stochastically dominated by $Z_{n+1}^{(j)}$. This settles {\bf ($\cH_2^{n+1}$)}. One may now show by iteration that the total number of particles in the branching process defined above dominates $|\cC(0)|$, which completes the proof.
\end{proof}

In the lemma below, $(N^k\one)(i)$ is the $i$-th term of the sequence $N^k\one$, $N^k$ being the $k$-th power of the (infinite) matrix $N$ and $\one$ being the (infinite) vector whose entries are all equal to one. In other terms, $(N^k\one)(i) = \sum_{j\in \N_0} N^k(i,j)$.
\begin{lemma}
\label{lemmaC}
If the series $\sum_{k\in\N_0} (N^k\one)(i)$ converges for all $i\in\N_0$, then $\tau_\ext$ is a.s.\ finite.
\end{lemma}

\begin{proof}
Note that $\sum_{k\in\N_0} (N^k\one)(i) = \E(\sum_{k\in\N_0} Z_k | Z_0 = i)$. If the latter is finite for all $i\in\N_0$ then $P(\sum_{k\in\N_0} Z_k < \infty | Z_0=i ) = 1$ for all $i\in\N_0$, which implies that 
$\sum_{k\in\N_0} Z_k$ is finite a.s. Thus, $\tau_\ext$ is a.s.\ finite.
\end{proof}

The lemma stated below provides an upper bound on $N(i,j)$. Its proof is quite technical and is therefore deferred to the end of the section.
\begin{lemma}\label{lem:UB.Ntr}
Let $d\geq 5$. Fix $\varepsilon>0$ and choose $t=\varepsilon r^{(4-d)/2}$. There exists $j_0\in\bbN$, $\ga>1$ and $r_0>0$ such that
\begin{equation}
N(i,j) \lesssim  \gep^2i^2 \one_{\{j\leq j_0\}} + i^\ga e^{-jt/2} \one_{\{j> j_0\}},\qquad i,j\in\bbN_0,\quad\mbox{ for all }r\leq r_0.
\end{equation}
\end{lemma}

We may now prove the lower bound in Theorem \ref{thm:5}.

\begin{proof}[Proof of the lower bound in Theorem \ref{thm:5}]
Let $t = \gep r^{(4-d)/2}$. We want to prove that there exists $\gep$ small enough such that, for all $r$ small enough, $\cO_{t,r}$ does not percolate.
Fix $i\in\N_0$. By Lemmas~\ref{lemmaA}--\ref{lemmaC}, it is enough to prove that for this choice of $t$ and $r$, the series $\sum_{k\geq 1} (N^k\one)(i)$ converges when $r$ is chosen small enough.
Note that there exists $r_0=r_0(\gep)$ small enough such that
\begin{equation}
\label{eq:largerC}
\sum_{j > j_0} j^{2\vee \ga} e^{-j t/2} \leq  (j_0+1)^{1+2\vee\ga} \gep^2, \qquad r\leq r_0.
\end{equation}
We are now going to prove that there exists $\tilde{r}_0=\tilde{r}_0(\gep)>0$ such that
\begin{equation}
\label{eq:iteration}
(N^k\one)(i) \leq 2 i^{2\vee\ga} (2c(j_0+1)^{1+2\vee\ga}\gep^2)^k, \qquad r\leq \tilde{r}_0,\quad k\in\N_0,
\end{equation}
where $c$ denotes the proportionality constant from Lemma~\ref{lem:UB.Ntr}.
To see that \eqref{eq:iteration} is true, note that by Lemma~\ref{lem:UB.Ntr},
\begin{equation}
\label{eq:iterationstart}
\sum_{j=0}^{\infty} N(i,j) \leq c(j_0+1)\varepsilon^2 i^2 + c\frac{i^{\alpha}}{1-e^{-t/2}} e^{-j_0t/2}
\end{equation}
By the choice of $t$ there is $r=r_1(\varepsilon) >0$ such that the second term on the right hand side of~\eqref{eq:iterationstart} is at most $c(j_0+1)i^{\alpha} \varepsilon^2$ for all $r\leq r_1$. Thus,~\eqref{eq:iteration} holds for $k=1$.
Assume now that \eqref{eq:iteration} has been proven for some $k\in\N$.  Then, for all $r\leq \tilde r_0 := r_0\wedge r_1$,
\begin{equation}
\label{eq:iterationend}
\begin{aligned}
(N^{k+1}\one)(i) &= \sum_{j=0}^{\infty} N(i,j)(N^{k}\one)(j)\\
&\leq (2c(j_0+1)^{1+2\vee\ga}\varepsilon^2)^k\sum_{j=0}^{\infty} 2j^{2\vee\ga}N(i,j)\\
&:= {\rm I} + {\rm II},
\end{aligned}
\end{equation}
where ${\rm I}$ and ${\rm II}$ equal the middle term in~\eqref{eq:iterationend} with the sum restricted to $j\leq j_0$ and $j>j_0$ respectively. An application of Lemma~\ref{lem:UB.Ntr} and Equation~\eqref{eq:largerC} shows the validity of~\eqref{eq:iteration} with $\tilde r_0 = r_0\wedge r_1$.
Choosing $\varepsilon >0$ such that $2c(j_0+1)^{1+2\vee\ga}\gep^2 <1$ yields the claim.
\end{proof}

We are left with proving Lemma~\ref{lem:UB.Ntr}. We shall makes use of the following lemma, whose proof is deferred to Section~\ref{S5.4}.
\begin{lemma}
\label{lem:ldestimate}
Let $d\geq 5$ and $r>0$. There exists $j_0$ large enough such that:
\be
\P(\cp(\ws{0,r}{[0,t]}) \geq j tr^{d-4}) \leq e^{-cjt},\qquad j\ge j_0.
\ee
\end{lemma}
\begin{proof}[Proof of Lemma \ref{lem:UB.Ntr}] 
We divide the proof in two parts: {\bf (1)} is the estimate for $j\leq j_0$ (which actually holds for all $j$), and {\bf (2)} is the estimate for $j> j_0$.

\noindent {\bf (1)} Let $A$ be a compact set. Define
\begin{equation}
\label{eq:NA}
N(A) = \bbE^0\Big[ |\{x\in \cE \colon B^x_{[0,t]}\cap A\neq \emptyset \}|\  \Big|\ B^0\ \Big].
\end{equation}
The reason why we condition on $B^0$ is that we later choose $A = \overline{W^{0,2r}_{[0,t]}}$. We first prove a general upper bound of the form
\begin{equation}
N(A) \lesssim t \cp(A) + \cp(A)^2.
\end{equation}
Campbell's formula yields
\begin{equation}
\label{eq:NACampbell}
\begin{aligned}
N(A) &= \int_{\bbR^d} \P(B^x_{[0,t]}\cap A\neq \emptyset)\, \dd x\\
&= \Leb_d(A) + \int_{\bbR^d\setminus A} \P(B^x_{[0,t]}\cap A\neq \emptyset)\,\dd x.
\end{aligned}
\end{equation}
Following Spitzer~\cite[Eq. (2.8)]{S64}, we obtain
\begin{equation}
\ba
\int_{\bbR^d\setminus A} \P(B^x_{[0,t]}\cap A\neq \emptyset)\, \dd x &\leq t\cp(A) + J(A),\quad
\\
&\mbox{where} \quad  J(A) = \int_{\bbR^d\setminus A} \P(B^x_{[0,\infty)}\cap A\neq \emptyset)^2\, \dd x.
\ea
\end{equation}
We are left with giving an upper bound on $J(A)$. To that end we write
\begin{equation}
\begin{aligned}
\label{eq:J(A)}
J(A)&=  \int_{d(x,A) \leq r} \P(B^x_{[0,\infty)}\cap A\neq \emptyset)^2\, \dd x+ 
\int_{d(x,A) > r} \P(B^x_{[0,\infty)}\cap A\neq \emptyset)^2\, \dd x\\
&= J_1(A)+J_2(A).
\end{aligned}
\end{equation}
We first derive an estimate on $J_2(A)$.
By~\eqref{eq:hittingprob} and \eqref{eq:green}, 
\begin{equation}
\P(B^x_{[0,\infty)}\cap A\neq \emptyset) = \kappa_d^{-1}\int_{y\in A} ||x-y||^{2-d} e_A(\dd y),
\end{equation}
where $e_A$ is the equilibrium measure of $A$.
By \eqref{eq:cap.e}, the measure $e_A(\dd y)/\cp(A)$ is a probability measure. Therefore, by Jensen's inequality,
\begin{equation}
\begin{aligned}
\label{eq:J2bound}
J_2(A) &=  \kappa_d^{-2}\int_{d(x,A) > r} \Big( \int_{y\in A} ||x-y||^{2-d} \frac{e_A(\dd y)}{\cp(A)} \Big)^2\,\cp(A)^2\, \dd x\\
&\lesssim \int_{y\in A} \int_{d(x,A) > r} ||x-y||^{4-2d}\,\dd x\, e_A(\dd y)\, \cp(A).\\
\end{aligned}
\end{equation}
Since $d(x,A)>r$ implies $||x-y||>r$ for all $y\in A$, we get
\begin{equation}
J_2(A) \lesssim \cp(A)^2 \times \int_{||x||\geq r}  ||x||^{4-2d}\, \dd x
\lesssim \cp(A)^2\, r^{4-d}.
\end{equation}
As for $J_1(A)$ we use the simple estimate $J_1(A) \leq \Leb_d(A\oplus \cB(0,r))$. We now replace $A$ by $\overline{W^{0,2r}_{[0,t]}}$. Using that $\overline{W^{0,2r}_{[0,t]}} \oplus \cB(0,r) = W^{0,3r}_{[0,t]}$ and the Poincaré-Faber-Szeg\"o inequality~\eqref{eq:PFSstatement}, we obtain 
\begin{equation}
J_1\Big(\overline{W^{0,2r}_{[0,t]}}\Big) \lesssim \cp(W^{0,3r}_{[0,t]})^{d/(d-2)} \lesssim r^{4-d}\cp(W^{0,3r}_{[0,t]})^2.
\end{equation}
To obtain the last inequality above, note that if $0\in\sC_i$ and for $r$ small enough
\be
\cp(W^{0,3r}_{[0,t]})^{d/(d-2) - 2} \lesssim (itr^{d-4})^{(4-d)/(d-2)} \lesssim (r^{4-d})^{(d-4)/(d-2)}\leq r^{4-d}. 
\ee

Adding up the upper bounds for $J_1$ and $J_2$, and using that $0\in \sC_i$, we may conclude this part of the proof.

{\bf (2)} Let $A$ be a measurable set and define
\begin{equation}
\label{eq:NjA}
N_j(A) = \bbE\Big[ \mid \{x\in\fC_j \colon B^x_{[0,t]} \cap A \neq \emptyset\} \mid \Big].
\end{equation}
By using first Campbell's formula and Cauchy-Schwarz in the probability inside the integral below, and noting that $\P(x\in\sC_j)$ does not depend on $x$,
\begin{equation}
\label{eq:NA(j)}
N_j(A) = \int_{\bbR^d} \P(x\in \sC_j,\ B^x_{[0,t]} \cap A \neq \emptyset)\, \dd x \leq \P(0\in\sC_j)^{1/2} \int_{\bbR^d} \P(B^x_{[0,t]} \cap A \neq \emptyset)^{1/2}\, \dd x.
\end{equation}
In {\bf (a)} and {\bf (b)} below we give an upper bound on $\P(0\in\sC_j)$ and $\int_{\bbR^d} \P(B^x_{[0,t]} \cap A \neq \emptyset)^{1/2}\, \dd x$, respectively.
\noindent {\bf (a)} By \eqref{def:classes} and Lemma \ref{lem:ldestimate},
\be
\P(0\in\sC_j) \leq \P(\cp(\ws{0, 3r}{[0,t]}) \geq j t r^{d-4}) \leq e^{-2jt}.
\ee

\noindent {\bf (b)} Recall the definition in \eqref{eq:outradius}. We write the integral in \eqref{eq:NA(j)} as $I_1+I_2$, where
\be
\label{eq:I1I2}
I_1 = \int_{||x||\leq 3M(A)} \P(B^x_{[0,t]} \cap A \neq \emptyset)^{1/2}\, \dd x \qquad \mbox{and} \qquad I_2 = \int_{||x||> 3M(A)} \P(B^x_{[0,t]} \cap A \neq \emptyset)^{1/2}\, \dd x.
\ee
Then, $I_1 \lesssim M(A)^d$. We now replace $A$ by $\ws{0, 2r}{[0,t]}$. Note that 
\be
\label{eq:volumeest}
\Leb_d(\ws{0,2r}{[0,t]}) \gtrsim r^{d-1} M(B^0_{[0,t]}),\qquad \mbox{thus} \qquad M(\ws{0,2r}{[0,t]})^d \lesssim r^{(1-d)d} \Leb_d(\ws{0,2r}{[0,t]})^d+r^d.
\ee
To see that the left hand inequality in \eqref{eq:volumeest} is true assume for simplicity that $M(B^{0}_{[0,t]}) = 4kr$ for some $k\in\N$. The extension to all other values of $M(B^{0}_{[0,t]})$ is straightforward. By continuity of Brownian motion in time, we may choose a sequence of points $x_i\in B_{[0,t]}^{0}$, $1\leq i\leq k$, such that $||x_i||=4ir$.
It readily follows that $\Leb_d(\bigcup_{i=1}^{k}\cB(x_i,2r))\gtrsim r^{d-1}M(B_{[0,t]}^{0})$ from which we may deduce the desired inequality with ease.
We now apply the Poincaré-Faber-Szeg\"o inequality~\eqref{eq:PFSstatement} to the right hand side of \eqref{eq:volumeest} and obtain
\be
\label{eq:PFS}
I_1 \lesssim r^{(1-d)d} \cp(\ws{0,2r}{[0,t]})^{d^2/(d-2)} +r^d.
\ee 
We now consider $I_2$. Note that $B^x_{[0,t]} \cap A \neq \emptyset$ and $||x||> 3M(A)$ imply that $B^x$ has travelled at least a distance $||x||/2$ before time $t$. Therefore,
\be\label{eq:I2bound}\ba
I_2 &\leq \int_{||x||> 3M(A)} \P\Big(\sup_{s\in [0,t]} || B^0_s || \ge || x ||/2\Big)^{1/2}\, \dd x\\
& \lesssim \int_{||x||> 3M(A)}  e^{-||x||^2/(8dt)}\, \dd x,\qquad \mbox{ by Doob's inequality,}\\
& \lesssim t^{d/2}.
\ea\ee
The last inequality follows by substituting $\overline{x}= x/\sqrt{t}$.
Finally, \eqref{eq:PFS}--\eqref{eq:I2bound} yield
\be
\label{eq:I1plusI2}
I_1 + I_2 \lesssim r^{(1-d)d} \cp(\ws{0, 2r}{[0,t]})^{d^2/(d-2)} +r^d + t^{d/2}.
\ee
Thus,
using $0\in\fC_i$, $t=\gep r^{(4-d)/2}$, \eqref{eq:I1plusI2} and \eqref{eq:NA(j)}--\eqref{eq:I1I2} we see that there exists an exponent $\ga>0$ and a constant $c_{\gep}$ such that
\be
N(i,j) \leq c_{\gep} (t(i+1))^\ga e^{-jt}.
\ee
However, we choose $r$ small enough such that $c_{\gep}e^{-jt/4}\leq 1$ and $t^{\alpha}e^{-jt/4}\leq 1$.
Hence, for this choice of $r$, we may write
\be
N(i,j) \lesssim i^\ga e^{-jt/2},\qquad j\geq j_0,
\ee
where the proportionality constant does not depend anymore on $\gep$. This concludes the proof.
\end{proof}

\subsection{Case $d=4$}
\label{S3.2}

We now turn to the case $d=4$. The idea of the proof is similar to the one in Section~\ref{S3.1}. However, the estimates in \eqref{eq:volumeest}--\eqref{eq:PFS} are not sufficiently sharp anymore. To overcome this difficulty, we choose a finer partition of $\cE$, namely
\be\label{eq:C4}
\sC_{i_1,i_2} = \Big\{x\in \cE \colon \cp(W_{[0,t]}^{x,3r}) \in [i_1,i_1+1)t/|\log r|,\, M(W_{[0,t]}^{x,2r})\in [i_2,i_2+1)t\Big\}, \qquad i_1, i_2 \in\bbN_0,
\ee
and define
\be\label{eq:N4}
 N(i_1,i_2;j_1,j_2)= \esssup\  \bbE^0 \Big[ \big| \sC_{j_1,j_2}\cap \cN(0)\big| \mid B^0 \Big]\one\{0\in \sC_{i_1,i_2}\},\qquad i_1,i_2,j_1,j_2\in\bbN_0.
\ee
The construction of the auxiliary branching process from \eqref{eq:C4} and \eqref{eq:N4} and the corresponding domination argument works along similar lines as in Section~\ref{S3.1}. We omit the details. The main difference to the proof in Section~\ref{S3.1} is that Lemma \ref{lem:UB.Ntr} needs to be replaced by Lemma~\ref{lem:UB.Ntr4} below.

\begin{lemma}\label{lem:UB.Ntr4}
Let $d=4$. Fix $\gep >0$ and $t = \gep \sqrt{|\log r|}$. There exists $j_0\in\bbN$ and $\ga>1$ such that for $r$ small enough
\begin{equation}
\ba
N(i_1,i_2;j_1,j_2) &\lesssim \gep^2 i_1^2 \log (i_1+1) \ind{j_1 \vee j_2 \leq j_0}\\
&+ i_2^4 t^4 (t^{-cj_1} \ind{j_2\leq j_0 < j_1} + e^{-j_2^2 t/2}\ind{j_1\leq j_0 < j_2} + t^{-cj_1/2}e^{-j_2^2t/4} \ind{j_1 \wedge j_2 > j_0})
\ea
\end{equation}
\end{lemma}
The proof of Lemma~\ref{lem:UB.Ntr4} is deferred to the end of this section. We first show how to deduce the lower bound in Theorem~\ref{thm:5} from it.

\begin{proof}[Proof of the lower bound in Theorem \ref{thm:5}]

Suppose $t = \gep \sqrt{|\log r|}$. We show that there exists $\gep$ small enough such that, for all $r$ small enough and for all $i_1,i_2\in\bbN_0$, the series $\sum_{k\geq 1} (N^k\one)(i_1,i_2)$ converges.
The proof idea is the same as for $d\geq5$. We only need to prove how to get from the estimate in Lemma \ref{lem:UB.Ntr4} to the convergence of the series $\sum_{k\geq 1} (N^k\one)(i_1,i_2)$. We assume that $t> 1$.
Throughout this proof we fix $j_0$ such that Lemma~\ref{lem:UB.Ntr4} is satisfied and such that for $r$ small enough
\be\label{eq:largeC4}
\max\Big( \sum_{j > j_0} (j^4 + j_0^4) t^{4-cj/2}, t^4 \sum_{j > j_0} (j^4 + j_0^4) e^{-j^2t/4} \Big) < \frac{1}{3} (j_0+1)^4\gep^2.
\ee
For these values of $r$ and all $\gep >0$ such that $4(j_0+1)^6\gep^2< 1$ we prove by iteration that for all $k\in\bbN$,
\be\label{eq:iteration4}
(N^k\one)(i_1,i_2) \leq (4(j_0+1)^6\gep^2)^k (i_1^4 + i_2^4),
\ee
which immediately yields the claim.
For $k=1$ this is a simple consequence of Lemma~\ref{lem:UB.Ntr4}. Assume that we have proven~\eqref{eq:iteration4} for some $k\in\N$. Then,
\begin{equation}
\label{eq:iterationend4}
\begin{aligned}
(N^{k+1}\one)(i_1,i_2) &=\sum_{j_1,j_2=0}^{\infty} N(i_1,i_2;j_1,j_2)N^k(j_1,j_2)\\
&\leq \sum_{j_1,j_2=0}^{\infty} N(i_1,i_2;j_1,j_2)(4(j_0+1)^6\gep^2)^k(j_1^4+j_2^4)\\
&:={\rm I + II + III + IV},
\end{aligned}
\end{equation}
where the terms ${\rm I-IV}$ equal the term in the second line with the sum restricted to
$\{j_1, j_2\leq j_0\}$, $\{j_2\leq j_0 < j_1\}$, $\{j_1\leq j_0 < j_2\}$ and $\{j_1\wedge j_2>j_0\}$ respectively. An application of Lemma~\ref{lem:UB.Ntr4} and Equation~\eqref{eq:largeC4} shows that~\eqref{eq:iteration4} holds for $k+1$ and hence yields the claim.
\end{proof}
For the proof of Lemma~\ref{lem:UB.Ntr4} we make use of the following lemma whose proof is given in Section~\ref{S5.4}.
\begin{lemma}
\label{ldestimate4d}
Let $d=4$. There exists $j_0\in(0,\infty)$ and $r_0>0$ such that for $j\geq j_0$ and $r\leq r_0$,
\be
\P\Big( \cp(\ws{0,r}{[0,t]}) \geq j \frac{t}{|\log r|}\Big) \leq t^{-cj}.
\ee
\end{lemma}

\begin{proof}[Proof of Lemma \ref{lem:UB.Ntr4}]
The proof is similar to the one of Lemma~\ref{lem:UB.Ntr}. Therefore we only sketch the proof and point out the main differences.
The proof is divided into four parts: {\bf (1)} $j_1, j_2\leq j_0$, {\bf (2)} $j_2\leq j_0 < j_1$, {\bf (3)} $j_1\leq j_0 < j_2$, and {\bf (4)} $j_1, j_2> j_0$.

\noindent {\bf (1)} Let $A$ be a compact set. We define $N(A)$ in analogy to~\eqref{eq:NA}. Thereafter we use Campbell's formula as in~\eqref{eq:NACampbell} and Spitzer~\cite[Eq. (2.8)]{S64} to obtain,
\be
\label{eq:intersect4}
N(A)\leq 
\Leb_d(A)+ t\cp (A)  + \int_{\bbR^4\setminus A} \P(B^x_{[0,t]}\cap A\neq \emptyset) \P(B^x_{[0,\infty)}\cap A\neq \emptyset)\, \dd x.
\ee

We decompose the domain of integration of the integral on the right hand side of \eqref{eq:intersect4} into three disjoints sets, which are given by,
\begin{equation}
\begin{aligned}
 &\{x\in\R^4:\, d(x,A) \leq r\},\,\quad   \{x\in\R^4:\, r < d(x,A) \leq 3(1+M(A))t\}, \\
 &\quad \mbox{ and } \{x\in\R^4:\, d(x,A) > 3(1+M(A))t\},
\end{aligned}
\end{equation}
and we denote the corresponding integrals by $J_1(A), J_2(A)$ and $J_3(A)$.
To bound the first of these three terms we use the trivial estimate $J_1(A) \leq \Leb_d(A\oplus \cB(0,r))^4$. As for $J_2$,  using similar estimates as in \eqref{eq:J2bound}, we obtain
\be
\label{eq:J24}
J_2(A)\lesssim \cp(A)^2(\log (M(A)+1)+|\log r|+\log t)\lesssim \cp(A)^2(\log(M(A)+1)+|\log r|)
\ee
Here, the second inequality follows from the fact that $\log t  \sim \log |\log r |$.
For the third part, we get using Doob's inequality
\be \ba
J_3(A) &\leq \int_{||x||>t} \P\Big(\sup_{0\leq s \leq t} ||B^0_s|| > ||x||/2\Big) \dd x\\
&\lesssim \int_{\rho>t} \rho^3 e^{-\rho^2/(8dt)}\dd\rho \lesssim e^{-t/64}.
\ea \ee
Choosing $A=\overline{\ws{0,2r}{[0,t]}}$, using that $0\in \fC_{i_1,i_2}$, and similar estimates as in~\eqref{eq:volumeest} as well as the Poincaré-Faber-Szeg\"o inequality~\eqref{eq:PFSstatement} to estimate the right-hand side of~\eqref{eq:J24} we get the claim.

For the next estimates {\bf (2), (3), (4)}, we start from the upper bound
\begin{equation}
\label{eq:NA4}
N_j(A) \leq \P(0\in\sC_{j_1;j_2})^{1/2} \int_{\bbR^d} \P(B^x_{[0,t]} \cap A \neq \emptyset)^{1/2}\, \dd x,
\end{equation}
where $N_j(A)$ is defined in a similar way as in~\eqref{eq:NjA}.
Using similar estimates as in \eqref{eq:I2bound}, we obtain
\be
\int_{\bbR^4} \P(B^x_{[0,t]} \cap A \neq \emptyset)^{1/2}\, \dd x \lesssim M(A)^4 + t^2.
\ee
Note that the latter term is bounded by $2i_2^4t^4$ if
$A = \overline{\ws{0,2r}{[0,t]}}$. Moreover, by Cauchy-Schwarz,
\be
\label{eq:Cj1j24}
\P(0\in\sC_{j_1;j_2}) \leq \P\Big(\cp(\ws{0, 3r}{[0,t]})\geq j_1\frac{t}{|\log r|}\Big)^{1/2} \P(M(\ws{0,r}{[0,t]})\geq j_2 t)^{1/2}.
\ee
We may bound $\P(\cp(\ws{0,3r}{[0,t]})\geq j_1\frac{t}{|\log r|})$ with the help of Lemma~\ref{ldestimate4d}.
As for $\P(M(\ws{0,r}{[0,t]})\geq j_2 t)$, we get by Doob's inequality
\be
\label{eq:radius4}
\P\Big(M(\ws{0,r}{[0,t]})\geq j_2 t\Big) \leq \P\Big(\sup_{s\in [0,t]} ||B_s|| \geq j_2 t - r \Big) \leq e^{-j_2^2 t/2}.
\ee
Hence, \eqref{eq:NA4}--\eqref{eq:radius4} and Lemma \ref{ldestimate4d} finish parts {\bf (2), (3)} and {\bf (4)} of the proof and thus complete the proof.
\end{proof}

%

\section{Proof of the upper bound}
\label{S4}

\par In this section we prove the upper bound in Theorem \ref{thm:5}, that is if we set
\begin{equation}\label{eq:param_upperbound}
t = c_* \times \left\{
\begin{array}{ll}
r^{(4-d)/2} & \quad \mbox{if } d\geq 5\\
\sqrt{\log(1/r)} & \quad \mbox{if } d=4
\end{array}
\right.,\quad r\in(0,1),
\end{equation}
then there exists $c_*$ large enough such that for $r$ small enough, $\cO_{t,r}$ percolates.\\

\par The proof is organized as follows. In Section \ref{S4.1}, we use a coarse-graining procedure to prove the existence of an unbounded component with a positive probability. More precisely, we divide space into boxes indexed by $\bbZ^d$ and we define a notion of {\it good} boxes, as well as a way to connect good boxes. Provided the box at the origin is good, we explore the cluster of good boxes connected to the origin and prove that with positive probability, this cluster is unbounded. This implies percolation. The procedure relies on two estimates, one on the probability for the box at the origin to be good (Lemma~\ref{lem:0good}), the other one on the probability of two neighbouring good boxes to be connected to each other (Lemma~\ref{lem:proba_no_connection}). These estimates are proven in Section~\ref{S4.2}.


\subsection{Coarse-graining procedure}
\label{S4.1}

\par Parameters are now chosen as in \eqref{eq:param_upperbound}. Let $c_B>0$ be a small constant to be determined later.
Let us consider the collection of disjoint balls $\cB_z = \cB(2zc_B\sqrt{t},c_B\sqrt{t})$, $z\in\bbZ^d$. In the following we identify $\bbZ^2\times\{0\}^{d-2}$ with $\bbZ^2$. We are going to prove that there is a choice of $c_B>0$ such that one may choose $c_*$ large enough and $r$ small enough such that percolation occurs by using only Wiener sausages from $\cup_{z\in\bbZ^2} \cB_z$.

We denote by $\cF_z$ the $\sigma$-algebra generated by the Poisson points in $\cB_z$ and their corresponding Brownian motions, and for $\gL\subseteq \bbZ^2$, $\cF_\gL = \bigvee_{z\in\gL} \cF_z$.\\

\par \noindent {\it Definitions.} Recall \eqref{eq:abb.int}. We define a set of {\it good} Poisson points by
\be
\cE_{\good} = \Big\{x\in A_{t/2}(c_B,r) \colon \cp(W^{x,r}_{[0,t/2]}) \geq \tfrac12 \E\big[\cp(W^{0,r}_{[0,t/2]}); 0\in A_{t/2}(c_B,r)\big]\Big\},
\ee
where 
\begin{equation}
\label{eq:A}
A_t(c,r)= \Big\{x\in \cE \colon W_{[0,t]}^{x,r}\subseteq \cB\Big(x,c\sqrt{t}\Big)\Big\}.
\end{equation}

\par \noindent {\it Construction of the cluster.} We now describe the algorithm we use to build a coarse-grained cluster. 
Before we start the construction of the cluster, we introduce the following order: $(1,0)\prec (0,1)\prec (-1,0)\prec (0,-1)$. We also use the convention that $C_{-1}=\emptyset$.\\
\textbf{Initialisation:} If $\cE_{\good}\cap\cB_0= \emptyset$, then set $C_0= \emptyset.$
Otherwise, set $D_0=\emptyset$, $c_0=0$ and $C_0= \{c_0\}$, and choose
$e_0\in\cE_{\good}\cap\cB_0$ such that $||e_0||=\min\{||x||:\, x\in\cE_{\good}\cap \cB_0\}$.\\
\noindent\textbf{Iteration:} Let $n\in\N_0$ and suppose $C_n, D_n\subseteq\Z^2$ with $C_n=\{c_i,\ 0\leq i\leq n\}$ as well as $e_i\in\cE_{\good}\cap \cB_{c_i}$ for $0\leq i\leq n$ are already constructed. The sets $C_n$ and $D_n$ represent the boxes already added to the cluster, respectively dismissed, at step $n$. We aim at defining $D_{n+1}$, $C_{n+1}$ and $e_{n+1}$. We distinguish between two cases.\\
\noindent \textbf{Case 1.} If $C_n=C_{n-1}$, then stop the iteration procedure.\\
\noindent \textbf{Case 2.} If $C_{n}\neq C_{n-1}$, define for all $0\leq i\leq n$,
\be
\cV_i=\{z\in\Z^2:\, |c_i-z|_1=1\colon \cE_{\good}\cap\cB_z\cap\cN(e_i)\neq \emptyset\} \quad \mbox{ and } \quad
\cV_i^{(n)}= \cV_i\cap (C_n\cup D_n)^c.
\ee
\noindent \textbf{Case 2a)} If $\cup_{0\leq i\leq n}\cV_i^{(n)} = \emptyset$, then set $C_{n+1}=C_n$ (consequently, the algorithm stops in the next step).\\
\noindent \textbf{Case 2b)} If $\cup_{0\leq i\leq n}\cV_i^{(n)} \neq \emptyset$, let
$i(n)= \max\{0\leq i\leq n:\, \cV_i^{(n)}\neq \emptyset\}$ and pick $c_{n+1} \in \cV_{i(n)}^{(n)}$ such that $c_{n+1}-c_{i(n)} = \min\Big\{z-c_{i(n)}:\, z\in \cV_{i(n)}^{(n)}\Big\}$. Additionally, pick $e_{n+1}\in\cE_{\good}\cap\cB_{c_{n+1}}\cap\cN(e_{i(n)})$ such that
$||e_{n+1}-e_{i(n)}||= \min\{||z-e_{i(n)}||:\, z\in\cE_{\good}\cap\cB_{c_{n+1}}\cap\cN(e_{i(n)})\}$.
We set
\be
\hat\cV_{i(n)}^{(n)} =\{z\in\Z^2:\, |c_{i(n)}-z|_1=1, \cE_{\good}\cap\cB_z\cap\cN(e_{i(n)})=\emptyset, z-c_{i(n)}\prec c_{n+1}-c_{i(n)}\}.
\ee
Finally, $C_{n+1}:=\{c_i,\ 0\leq i\leq n+1\}$ and
\be
 \quad D_{n+1}:=D_n\cup \hat\cV_{i(n)}^{(n)} \cup\{c_i+z:\, i(n)<i\leq n, |z|_1=1, c_i+z\in (C_n\cup D_n)^c\}.
\ee
This finishes the description of the algorithm.

\par If the algorithm does not stop, it means that $\cO_{t,r}$ contains an unbounded component. If it stops at step $n$, we denote by $\bC$ the set of connected boxes $C_n$ obtained in this way. Therefore, we are going to prove that the algorithm stops with probability strictly less than one. \\

\par For the rest of the proof we rely on the following two key lemmas, which will be proven in Section \ref{S4.2}. For convenience, we say that $z\in\bbZ^2$ is good if $\cB_z$ contains a point in $\cE_{\good}$.
\begin{lemma}
\label{lem:0good}
Let $d\geq 4$ and fix $c_B>0$. The probability that $0$ is good converges to $1$ as $t$ goes to $\infty$.
\end{lemma}
Also, we say that $z'\in\bbZ^2$ is connected to $z\in \bbZ^2$ if there exist $x'\in\cB_{z'}\cap \cE_{\good}$ and $x\in\cB_z\cap \cE_\good$ such that $x\sim x'$ (recall \eqref{eq:abb.int}).
For a set $\Lambda\subseteq\Z^2$ we say that $z'\in\Z^2$ is connected to $\Lambda$, if there is $z\in\Lambda$ with $|z-z'|_1=1$ such that $z$ is connected to $z'$.
\begin{lemma}
\label{lem:proba_no_connection}
Let $d\geq 4$, fix $c_B>0$, let $z,z'\in\bbZ^2$ such that $|z-z'|_1 = 1$. On the event $\{z \mbox{ is good}\}$, we have for $t$ large enough,
\begin{equation}
\label{eq:proba_no_connection}
\bbP(z' \mbox{ is not connected to } z | \cF_z) \leq \exp\{-c_*^2 \theta(c_B)\},\quad \mbox{with} \quad \liminf_{c_B \to 0}\theta(c_B) >0.
\end{equation}
\end{lemma}

\begin{proof}[Proof of the upper bound in Theorem \ref{thm:5}]
We now explain how to conclude the proof with these two lemmas at hand. For this, we use the so-called standard Peierls contour argument, see Grimmett~\cite[Proof of Theorem 1.10]{Gr99}. In what follows, a $*$-path of length $N\geq2$ is a vector $(x_i)_{1\leq i \leq N} \in (\bbZ^2)^N$ such that $|x_{i+1} - x_i|_{\infty} = 1$ for all $1\leq i < N$. If $x_N = x_1$ and for all $1\leq i, j<N$ with $i\neq j$, $x_i \neq x_j$, then the $*$-path is said to be a $*$-contour. This contour contains $x\in\bbZ^2$ if $x$ belongs to the bounded component delimited by the contour, but not to the contour itself. We denote the set of all contours containing $x\in\bbZ^2$ by $\bbZ^2_{\mathrm{con}}(x)$. Denote by $\partial_\ext\bC$ the exterior boundary of $\bC$, that is
the set of vertices in the boundary which are the starting points of an infinite non-intersecting nearest neighbor path with no vertex in $\bC$. By Grimmett~\cite[p17]{Gr99} (see also the reference to Kesten~\cite{K82} therein for more details) we see that if $|\bC|<\infty$, then $\partial_\ext\bC$ is a $*$-contour. We may write
\begin{equation}
\label{eq:contour}
\bbP(|\bC| < \infty) \leq \bbP(0 \mbox{ is not good}) + \sum_{N\geq4} \bbP(|\partial_{\ext}\bC| = N).\\
\end{equation}

Let us give an upper bound on $\bbP(|\partial_{\ext}\bC| = N)$. We have
\begin{equation}
\bbP(|\partial_{\ext}\bC| = N) = \sumtwo{\gL\in\bbZ^2_{\mathrm{con}}(0),}{|\gL|=N} \bbP(\partial_{\ext}\bC = \gL),
\end{equation}
and for each such $*$-contour $\gL$,
\begin{equation}
\ba
\bbP(\partial_{\ext}\bC = \gL) &\leq \bbP(\forall z\in \gL,\ z  \mbox{ not connected to } \partial_{\inte}\gL) \\
&= \bbE[ \bbP(\forall z\in\gL,\ z\mbox{ not connected to } \partial_{\inte}\gL | \cF_{\partial_{\inte}\gL}) ],
\ea
\end{equation}
where $\partial_{\inte}\gL = \partial\gL \setminus \partial_{\ext}\gL$ and $\partial\gL = \{z\notin \gL \colon \exists z'\in \gL / |z-z'|_1=1\}$. Since the events $\{z \mbox{ not connected to } \partial_{\inte}\gL\}_{z\in\gL}$ are independent conditionally on $\cF_{\partial_{\inte}\gL}$, we get
\begin{equation}
\bbP(\forall z\in\gL,\ z\mbox{ not connected to } \partial_{\inte}\gL | \cF_{\partial_{\inte}\gL}) = \prod_{z\in\gL} \bbP( z\mbox{ not connected to } \partial_{\inte}\gL | \cF_{\partial_{\inte}\gL})
\end{equation}
Let us fix $z\in\gL$ and denote by $z'$ an $\ell_1$-neighbour of $z$ which is also in $\partial_\inte \gL$. No matter how we choose $z'$, we get
\begin{equation}
\ba
\bbP( z\mbox{ not connected to } \partial_{\inte}\gL | \cF_{\partial_{\inte}\gL}) &\leq \bbP( z\mbox{ not connected to } z' | \cF_{\partial_{\inte}\gL})\\
&= \bbP( z\mbox{ not connected to } z' | \cF_{z'}),
\ea
\end{equation}
which is smaller than $e^{-c_*^2\theta(c_B)}$, by Lemma~\ref{lem:proba_no_connection}. Therefore, we get
\begin{equation}
\bbP(|\bC| < \infty) \leq \bbP(0\mbox{ is not good}) + \sum_{N\geq4} e^{-c_*^2\theta(c_B)N} C_N, 
\end{equation}
where $C_N$ is the number of $*$-contours of length $N$ containing the origin.
By a standard counting argument (see Grimmett~\cite[Proof of Theorem 1.10]{Gr99}) it can be seen that $C_N \leq N\ 7^N$. We obtain
\begin{equation}\label{eq:prob_cluster_finite}
\bbP(|\bC| < \infty) \leq \bbP(0 \mbox{ is not good}) + c\,\sum_{N\geq4} N\ \Big(7e^{-c^2_*\theta(c_B)}\Big)^N.
\end{equation}
We conclude as follows. First, fix $c_B$ small enough such that $\theta(c_B)$ is positive, see \eqref{eq:proba_no_connection}. Then, choose $c_*$ so large that the sum in the r.h.s of \eqref{eq:prob_cluster_finite} is smaller than $1/(4c)$. Finally, choose $r$ small enough (therefore $t$ large enough) such that, by Lemma \ref{lem:0good}, $\bbP(0\mbox{ is not good})\leq 1/4$. This finally yields $\bbP(|\bC| = \infty) \geq 1/2$, which finishes the proof.
\end{proof}

\subsection{Proof of Lemmas \ref{lem:proba_no_connection} and \ref{lem:0good}.}
\label{S4.2}

\par Throughout this section we shall make use of the capacity estimates provided by Lemmas~\ref{lem:uppercap2}--\ref{lem:5lowercap} below. Lemma~\ref{lem:uppercap2} gives an estimate on the second moment of the capacity of a Wiener sausage, whereas Lemma~\ref{lem:5lowercap} estimates the mean capacity of a Wiener sausage confined to a ball with radius of order $\sqrt{t}$. Their proofs are deferred to Section \ref{S5}.
\begin{lemma}\label{lem:uppercap2}
Let $d\geq 4$, $t_0>1$ and $r_0\in (0,1)$. For all $t\geq t_0$ and all $r\in (0,r_0)$,
\begin{equation}
\label{eq:uppercap}
\E\Big[\cp\big(W_{[0,t]}^{0,r}\big)^2\Big] \lesssim \left\{
\begin{array}{ll}
t^2\,r^{2(d-4)} & \quad \mbox{if } d\geq 5 \\
\left(\frac{t}{\log(tr^{-2})}\right)^2 & \quad \mbox{if } d=4.
\end{array}
\right.
\end{equation}
\end{lemma}
\begin{lemma}
\label{lem:5lowercap}
Recall \eqref{eq:A}. For all $d\geq4$, $t\geq 1$ and $r\in (0,1)$,
\begin{equation}
\label{eq:5lowercap}
\E\Big[\cp\Big(W_{[0,t]}^{0,r}\Big);0\in A_t(c_B,r)\Big] \gtrsim \P(0\in A_t(c_B,r))^2\ \times
\left\{
\begin{array}{ll}
t\ r^{d-4} & \quad \mbox{if } d\geq5\\
\frac{t}{\log(tr^{-2})} & \quad \mbox{if } d=4.
\end{array}
\right.
\end{equation}
\end{lemma}

\par We start with Lemmas \ref{lem:ugoodcap} and \ref{lem:goodintersect}, which are preparatory lemmas. Lemma \ref{lem:ugoodcap} gives a lower bound on the probability that a Wiener sausage has a capacity larger than a fraction of its mean capacity, when it is confined to a ball of order $\sqrt{t}$. Lemma \ref{lem:goodintersect} gives a lower bound on the probability that a Wiener sausage intersects a set that is at a distance of order $\sqrt{t}$ from its starting point.

\begin{lemma}
\label{lem:ugoodcap}
Let $d\geq 4$. Abbreviate by $\cA$ the event $\{0\in A_t(c_B,r)\}$, see \eqref{eq:A}. Then,
\begin{equation}
\label{eq:ugoodcap}
\P\Big(\Big\{\cp\Big(W_{[0,t]}^{0,r}\Big)\geq \tfrac12\ \E\Big(\cp\Big(W_{[0,t]}^{0,r}\Big) ; \cA\Big)\Big\} \cap \cA\Big)\gtrsim \Phi(c_B)^4(1+o(1)),
\end{equation}
where $\Phi(c_B) = \P\big(\sup_{s\in[0,1]} ||B^{0}_s||\leq c_B\big)$ and the $o(1)$ term tends to zero as $t$ tends to infinity.
\end{lemma}

\begin{proof}
By (a slight generalization of) the Paley-Zigmund inequality,
\begin{equation}
\label{eq:PZ}
\P\Big(\Big\{\cp\Big(W_{[0,t]}^{0,r}\Big)\geq \tfrac12\ \E\Big(\cp\Big(W_{[0,t]}^{0,r}\Big) ; \cA\Big)\Big\} \cap \cA\Big)
\geq \tfrac14 \frac{\E\Big[\cp\Big(W_{[0,t]}^{0,r}\Big);\cA\Big]^2}
{\E\Big[\cp\Big(W_{[0,t]}^{0,r}\Big)^2\Big]}.
\end{equation}
Using Lemma~\ref{lem:uppercap2} and Lemma~\ref{lem:5lowercap}, and since by invariance of Brownian motion,
\begin{equation}
\label{eq:PofA}
\P(\cA) = \P\Big(W_{[0,1]}^{0,r/\sqrt{t}}\subseteq \cB(0,c_B)\Big)=
\P\Big(\sup_{s\in[0,1]} ||B^{0}_s||\leq c_B-\frac{r}{\sqrt{t}}\Big)=\Phi(c_B)(1+o(1)),
\end{equation}
we get the claim.
\end{proof}

Given a measurable set $A\subseteq \R^d$ we write
\begin{equation}
\label{eq:rthick}
A^r= A \oplus \cB(0,r).
\end{equation}
\begin{lemma}
\label{lem:goodintersect}
There is a constant $c\in(0,\infty)$ such that the following estimate
holds uniformly for all $r\in(0,1)$ and all measurable sets $A$ such that $A\subseteq \cB(0,6c_B\sqrt{t})$,
\begin{equation}
\label{eq:intersectest}
\P\Big(W_{[0,t]}^{0,r}\cap A\neq \emptyset\Big)
\geq t^{1-d/2}\cp(A^r)\bigg(\frac{c}{c_B^{d-2}}-\frac{1}{(2\pi)^{d/2}}\bigg).
\end{equation}
\end{lemma}
\begin{proof}
Note that
\begin{equation}
\label{eq:errorterm}
\begin{aligned}
\P&\Big(W_{[0,t]}^{0,r}\cap A\neq \emptyset\Big)\\
&= \P\Big(W_{[0,\infty)}^{0,r}\cap A\neq \emptyset\Big)
-\P\bigg(\inf\Big\{s>0\colon W_{[0,s]}^{0,r}\cap A\neq \emptyset\Big\} \in(t,\infty)\bigg),
\end{aligned}
\end{equation}
so that it is enough to find a lower bound for the first term on the right hand side of \eqref{eq:errorterm} and an upper bound for the second term on the right hand side of \eqref{eq:errorterm}.
Let $e_{A^r}$ be the equilibrium measure of $A^r$. The identity in \eqref{eq:hittingprob} yields
\begin{equation}
\label{eq:firstterm}
\P\Big(W_{[0,\infty)}^{0,r}\cap A\neq \emptyset\Big)
= \P\Big(B_{[0,\infty)}^{0}\cap A^r\neq \emptyset\Big)
=\int_{A^r} G(0,y) e_{A^r}(dy).
\end{equation}
Hence, using that $G(0,y) = c\,||y||^{2-d}$ and $e_{A^r}(A^r)= \cp(A^r)$, \eqref{eq:firstterm} may be bounded from below by
\begin{equation}
\label{eq:firstbound}
c\inf_{y\in A^r}||y||^{2-d}\cp(A^r).
\end{equation}
Since $A^r\subseteq \cB(0,6c_B\sqrt{t}+r)$, we see that 
there is a constant $c>0$ such that \eqref{eq:firstbound} is at least
\begin{equation}
\label{eq:firstboundfinal}
c\ t^{1-d/2} c_B^{2-d}\cp(A^r), \qquad r\leq 1.
\end{equation}
This is the desired lower bound for the first term on the right hand side of \eqref{eq:errorterm}.
Recall that $\tilde{B}$ is a Brownian motion independent of $B^0$.
By the Markov property and \eqref{eq:hittingprob}, the second term on the right hand side
of \eqref{eq:errorterm} may be written as
\begin{equation}
\label{eq:hitaftert}
\begin{aligned}
\E&\bigg[\one\Big\{B_{[0,t]}^{0}\cap A^r=\emptyset\Big\}
\tilde{\P}_{{B_t^{0}}}\Big(\tilde{B}_{[0,\infty)}\cap A^r\neq \emptyset\Big)\bigg]\\
&= \E\bigg[\one\Big\{B_{[0,t]}^{0}\cap A^r=\emptyset\Big\}
\int_{A^r}G(B_t^0,y)\, e_{A^r}(\dd y)\bigg].
\end{aligned}
\end{equation}
Hence,
\begin{equation}
\label{eq:hitaftertupper}
\eqref{eq:hitaftert} \leq \int_{A^r}\E[G(B_t^0,y)]\, e_{A^r}(\dd y).
\end{equation}
We obtain by the Markov property applied to $B^0$ at time $t$,
\begin{equation}
\label{eq:truncG}
\E(G(B_t^{0},y))
= \int_{t}^{\infty}\P(B^{0}_s\in \dd y)\, \dd s
= \int_{t}^{\infty}\frac{1}{(2\pi s)^{d/2}} e^{-\|y \|^2/2s}\, \dd s.
\end{equation}
Using the change of variable $w=||y||^2/2s$, we see that
\begin{equation}
\label{eq:truncgamma}
\eqref{eq:truncG} = \int_0^{\|y\|^2/2t} w^{d/2-2}e^{-w}\, \dd w \times\frac{\|y\|^{2-d}}{2\pi^{d/2}},
\end{equation}
which is bounded from above by $(2\pi)^{-d/2}t^{1-d/2}$ (by bounding the exponential factor by $1$). Therefore,
\begin{equation}
\label{eq:secondtermest}
{\rm r.h.s.} \eqref{eq:hitaftertupper} \leq \cp(A^r)t^{1-d/2}(2\pi)^{-d/2}.
\end{equation}
Combining \eqref{eq:errorterm} with \eqref{eq:firstboundfinal} and \eqref{eq:secondtermest} yields the claim.
\end{proof}

\begin{proof}[Proof of Lemma \ref{lem:proba_no_connection}]
Let $z,z'\in \bbZ^2$ with $|z-z'|_1 = 1$.
Let us abbreviate $\bar\bbP(\cdot) = \bbP(\cdot\ |\ \cF_z)$ and note that we are on the event $\cE_{\good}\cap\cB_z\neq\emptyset$. Let $x\in \cE \cap \cB_{z'}$. We first give a lower bound on the probability that $x$ is good and connected to a point in $\cE_{\good}\cap\cB_z$, that is
\begin{equation}
p_0 := \inf_{y\in \cE_{\good}\cap\cB_z}\bar\bbP(x\in \cE_{\good} ,\ x\sim y).
\end{equation}
Using the Markov property on $B^x$ at time $t/2$ and that $B^x_{t/2}\in \cB(y,5c_B\sqrt{t})$ this probability can be bounded from below by
\begin{equation}
\P\Big(x\in\cE_{\good}\Big)\times
 \inftwo{x_0\in \cB(y, 5c_B \sqrt{t})}{y\in \cE_{\good}\cap\cB_z}\bar\bbP\Big(W^{x_0,r}_{[0,t/2]} \bigcap W^{y,r}_{[0,t/2]} \neq \emptyset\Big).
\end{equation}
Using Lemma \ref{lem:ugoodcap} on the first factor and Lemma \ref{lem:goodintersect} on the second factor and noticing that for all $x_0\in \cB(y,5c_B\sqrt{t})$ and $y\in \cE_{\good}\cap\cB_z$, $W^{y,r}_{[0,t/2]} \subseteq \cB(x_0,6c_B\sqrt{t})$, we get that this probability is larger than
\begin{equation}
 \phi(c_B) t^{-d/2+1}\inf_{y\in \cE_{\good}\cap\cB_z} \cp\Big(W^{y,r}_{[0,t/2]}\Big).
\end{equation}
Here, $\phi(c_B) := c\ \Phi(c_B)^4 \bigg(\frac{c}{c_B^{d-2}}-\frac{1}{(2\pi)^{d/2}}\bigg)(1+o(1))$ is positive provided $c_B$ is small enough. By definition, we know that for all $y\in \cE_{\good}\cap\cB_z$
\begin{equation}
\cp\Big(W^{y,r}_{[0,t/2]}\Big)\geq \tfrac12 \E\Big[\cp\Big(W_{[0,t/2]}^{0,r}\Big); 0\in A_{t/2}(c_B,r)\Big].
\end{equation}
Recalling \eqref{eq:param_upperbound}, \eqref{eq:PofA} and Lemma \ref{lem:5lowercap}, we obtain
\begin{equation}
p_0 \geq c\, c_*^2 \phi(c_B)\Phi(c_B)^2 t^{-d/2}(1+o(1)). 
\end{equation}
Therefore, for all $y\in \cE_{\good}\cap\cB_z$, the number of points in $\cE_{\good}\cap\cB_{z'}$ connected to $y$ is a Poisson random variable with parameter bounded from below by:
\begin{equation}
p_0 \times \Leb(\cB_{z'}) \geq c_*^2 \ \theta(c_B),\qquad \mbox{with} \quad \theta(c_B) = c\ c_B^d \phi(c_B)  \Phi(c_B)^2,
\end{equation}
which is uniform in $y$ and concludes the proof.
\end{proof}

\begin{proof}[Proof of Lemma \ref{lem:0good}]
If $x\in \cE\cap\cB_0$ then the probability that $x$ is good is larger than $c\ \Phi(c_B)^4(1+o(1))$, by Lemma \ref{lem:ugoodcap}. Therefore, the number of such points is a Poisson random variable with parameter bounded from below by $c\ c_B^d\ \Phi(c_B)^4\ t^{d/2}(1+o(1))$, which goes to $\infty$ as $t\to\infty$. This concludes the proof.
\end{proof}

\section{Capacity estimates}
\label{S5}

\subsection{Green function estimates}
\label{S5.1}

\begin{lemma}
\label{lem:Greenbound}
 Let $d\geq 4$ and $t_0>1$. For all $t \geq t_0$,
\begin{equation}
\label{eq:Greenbound}
\E\Bigg[\int_{[0,t]^2}\int_{\cB(0,1)^2}
G(B^{0}_u + z, B^{0}_v + z')\,  \dd z\, \dd z'\, \dd u\, \dd v \Bigg]
\lesssim \left\{ \begin{array}{ll}  t, &  \quad \mbox{if } d\geq 5,\\  t \log t, &  \quad \mbox{if } d= 4. \end{array} \right.
\end{equation}
\end{lemma}

\begin{proof}
\noindent {\bf Case $d\geq 5$.}\\
We start with two estimates.
First, let $0\leq u\leq 1$. We claim that
\begin{equation}
\label{eq:1stclaim}
\E(G(B_u^{0},z))\leq G(0,z)\quad \mbox{for all }z\in\R^d.
\end{equation}
Indeed, an application of the Markov property in the second equality below shows that
\begin{equation}
\label{eq:smallu}
\E(G(B_u^{0},z)) = \int_0^{\infty}\E[\tilde{\P}_{B_u^0}(\tilde{B}_s\in \dd z)]\, \dd s
= \int_0^{\infty}\P(B_{u+s}^{0}\in \dd z)\, \dd s
= \int_{u}^{\infty} \P(B_s^0\in \dd z)\, \dd s.
\end{equation}
Since the right hand side is bounded from above by $G(0,z)$ we obtain \eqref{eq:1stclaim}.
Now, let $u>1$. In this case we claim that
\begin{equation}
\label{eq:2ndclaim}
\E(G(B_u^{0},z))\leq c\,u^{1-d/2}\qquad z\in\R^d.
\end{equation}
This is a direct consequence of \eqref{eq:truncG}--\eqref{eq:truncgamma}.
To make use of the inequalities \eqref{eq:1stclaim} and \eqref{eq:2ndclaim} we write the left hand side in \eqref{eq:Greenbound} as a sum of three terms:
\begin{equation}
\label{eq:split}
\begin{aligned}
(1)&= \E\Bigg[\int_{0}^{t}\int_{0}^{(v-1)\vee 0}\int_{\cB(0,1)^2}
G(B^{0}_u + z, B^{0}_v + z')\,  \dd z\, \dd z'\, \dd u\, \dd v \Bigg],\\
(2)&= \E\Bigg[\int_{0}^{t}\int_{(v+1)\wedge t}^{t}\int_{\cB(0,1)^2}
G(B^{0}_u + z, B^{0}_v + z')\,  \dd z\, \dd z'\, \dd u\, \dd v \Bigg],\\
(3)&= \E\Bigg[\int_{0}^{t}\int_{(v-1)\vee 0}^{(v+1)\wedge t}\int_{\cB(0,1)^2}
G(B^{0}_u + z, B^{0}_v + z')\,  \dd z\, \dd z'\, \dd u\, \dd v \Bigg].
\end{aligned}
\end{equation} 
We first estimate the third term.
Note that for all $x,y\in\R^d$ the relation $G(x,y)=G(0,y-x)$ holds. Hence, a change in the order of integration together with equation \eqref{eq:1stclaim} and the fact that $B_v^0-B_u^0$ has the same distribution as $B_{|v-u|}^0$ show that
\begin{equation}
\label{eq:polarterm}
(3)\leq \int_{\cB(0,1)^2}\int_{0}^{t}\int_{v-1}^{v+1} G(0,z-z')\, \dd u\, \dd v\,\dd z\, \dd z'\leq 2t\int_{\cB(0,1)^2}G(0,z-z')\, \dd z\, \dd z'.
\end{equation}
Hence, it suffices to show that the integral on the right-hand side of \eqref{eq:polarterm} converges.
By \eqref{eq:green}, the right hand
of \eqref{eq:polarterm} is at most
\begin{equation}
\label{eq:geometric}
2ct\int_{\cB(0,1)}\int_{\cB(z',1)} \|z-z'\|^{2-d} \dd z\, \dd z' 
= 2ct\int_{\cB(0,1)^2}\|z\|^{2-d}\, \dd z\, \dd z',
\end{equation}
where we made the substitution $\zeta=z-z'$ to obtain the last equality.
Since the integral on the right-hand side of \eqref{eq:geometric} is finite, $(3)\leq ct$.
It remains to show that the first and second terms in \eqref{eq:split} give the correct contribution.
Equation \eqref{eq:2ndclaim} yields
\begin{equation}
\label{eq:estof1}
(1)\leq c\ \int_{\cB(0,1)^2}\int_{0}^{t}\int_{0}^{(v-1)\vee 0}
|v-u|^{1-d/2}\, \dd u\, \dd v\, \dd z\, \dd z'.
\end{equation}
A simple computation now shows that there is indeed a constant $c>0$ such that for all
$t\geq 0$ the bound $(1)\leq ct$ holds. The argument for (2) in \eqref{eq:split} is similar and will therefore be omitted.
This finishes the proof in this case.\\
\noindent {\bf Case $d=4$.}
The proof works almost verbatim as in the previous case. The only difference is that \eqref{eq:estof1} becomes
\begin{equation}
\label{eq:4estof1}
\int_{\cB(0,1)^2}\int_0^{t}\int_0^{(v-1)\vee 0} |v-u|^{-1}\, \dd u\, \dd v\, \dd z\, \dd z',
\end{equation}
which is upper bounded by $ct\log t$.
We omit the details.
\end{proof}

\subsection{Lower bounds. Proof of Lemma \ref{lem:5lowercap}.}
\label{S5.2}

The proof of Lemma~\ref{lem:5lowercap} makes use of the variational representation in \eqref{eq:cap}, according to which it suffices to construct a measure which is close to the "true" minimizer in \eqref{eq:cap}. It will turn out that it is enough to choose a measure of the local time of the Brownian motion in a neighborhood of a given set. In this way the Green function estimates of Lemma \ref{lem:Greenbound} enter naturally into the picture. 

\begin{proof}[Proof of Lemma \ref{lem:5lowercap}]
We start with the case $d\geq 5$.\\
\textbf{1st Step:} Let $r=1$ and $\nu$ be the probability measure supported on $W_{[0,t]}^{0,1}$ and defined by
\begin{equation}
\label{eq:defnu}
\nu(A) = \frac{1}{c_{\mathrm{vol}}t} \int_{0}^{t}\int_{\cB(0,1)}
\one\{B_s^{0} + z\in A\}\, \dd s\, \dd z,\quad A\mbox{ Borel-measurable}.
\end{equation}
Note that by the variational formula in \eqref{eq:cap},
\begin{equation}
\label{eq:caplower}
\E\Big[\cp\Big( W_{[0,t]}^{0,1}\Big);0\in A_t(c_B,1)\Big]
\geq \E[\cI(\nu)^{-1}; 0\in A_t(c_B,1)],
\end{equation}
where
\begin{equation}
\label{eq:energynu}
\cI (\nu) = \frac{1}{c_{\mathrm{vol}}^2 t^2}\int_{[0,t]^2}\int_{\cB(0,1)^2} 
G(B_u^{0}+z, B_v^{0} + z')\, \dd z\, \dd z' \dd u\, \dd v.
\end{equation}
By the Cauchy-Schwarz inequality,
\begin{equation}
\label{eq:Jensenbound}
\E[\cI(\nu)^{-1};0\in A_t(c_B,1)]\geq \E[\cI(\nu)]^{-1}\P(0\in A_t(c_B,1))^2.
\end{equation}
Finally, by Equation~\eqref{eq:Greenbound}, the right hand side of \eqref{eq:Jensenbound} is bounded from below by $ct\P(A_t(c_B,1))^2$. This yields the claim in the case $r=1$.\\
\noindent
\textbf{2nd Step:} Let now $r>0$ be chosen arbitrarily.
By Brownian scaling and the capacity scaling relation \eqref{eq:scaling},
\begin{equation}
\label{eq:generalr}
\begin{aligned}
\E\Big[\cp\Big(W_{[0,t]}^{0,r}\Big);0\in A_t(c_B,r)\Big]
&= \E\Big[\cp \Big(\frac{r}{r}W_{[0,t]}^{0,r}\Big);0\in A_t(c_B,r)\Big]\\
&= r^{d-2} \E\Big[\cp \Big(W_{[0,t/r^2]}^{0,1}\Big);0\in A_{tr^{-2}}(c_Br^{-1},1)\Big].
\end{aligned}
\end{equation} 
Using the result for the case $r=1$ and noting that $\P(0\in A_{tr^{-2}}(c_Br^{-1},1))=\P(0\in A_t(c_B,r))$ finishes the proof for $d\geq 5$.

The proof in the case $d=4$ works along similar lines, the only difference being that the application of Lemma~\ref{lem:Greenbound} is adapted.
\end{proof}

\subsection{Second moment estimates. Proof of Lemma \ref{lem:uppercap2}}

\label{S5.3}
\subsubsection{Case $d\geq 5$}
\label{S5.3.1}
\begin{proof}
\textbf{1st Step:} In this step we prove Lemma~\ref{lem:uppercap2} under the assumption 
$r=1$. 
First note that by Equation \eqref{eq:unionbound}
\begin{equation}
\label{eq:capseparation}
\cp\Big(W_{[0,t]}^{0,1}\Big) \leq 
\cp\Bigg(\bigcup_{i=1}^{\lceil t \rceil} W_{[(i-1), i]}^{0,1}\Bigg)  
\leq \sum_{i=1}^{\lceil t \rceil} \cp\Big( W_{[(i-1), i]}^{0,1}\Big),
\end{equation}
so that by the independence of $B_{i}^{0} - B_{i-1}^{0}$ and $B_{j}^{0} - B_{j-1}^{0}$ for all $i \neq j$ in $\{1,2,\ldots,\lceil t\rceil\}$,
\begin{equation}
\label{eq:squaresum}
\begin{aligned}
&\E\Big[\cp\Big(W_{[0,t]}^{0,1}\Big) ^2\Big]\\
&\leq \sum_{\substack{i,j=1\\i\neq j}}^{\lceil t\rceil}
\E\Big[\cp\Big(W_{[(i-1), i]}^{0,1}\Big)\Big]
\times\E\Big[\cp\Big(W_{[(j-1), j]}^{0,1}\Big)\Big]
+\sum_{i=1}^{\lceil t\rceil} \E\Big[\cp\Big(W_{[(i-1), i]}^{0,1}\Big)^2\Big].
\end{aligned}
\end{equation}
Consequently, by the stationarity in time of Brownian motion and by the Cauchy-Schwarz inequality, the right hand side of \eqref{eq:squaresum} is bounded from above by
\begin{equation}
\label{eq:squarebound}
\lceil t\rceil ^2 \times \E\Big[\cp\Big(W_{[0,1]}^{0,1}\Big)^2\Big].
\end{equation}
To see that the expectation on the right hand side of \eqref{eq:squarebound} is finite,
note that by the scaling relation \eqref{eq:scaling}, $\cp(\cB(0,R))= R^{d-2}\cp(\cB(0,1))$ for any $R>0$. Since $\E(\sup_{s\in (0,1)} \|B_s\|^{d-2})<\infty$, the desired finiteness readily follows. This proves Lemma~\ref{lem:uppercap2} in the case $r=1$.\\
\noindent
\textbf{2nd Step:} We now treat the general case.
To that end, note that by Brownian scaling and by the scaling relation \eqref{eq:scaling},
\begin{equation}
\label{eq:capscaling}
\cp\Big(W_{[0,t]}^{0,r}\Big)= \cp\Big(\frac{r}{r} W_{[0,t]}^{0,r}\Big)
\stackrel{\mbox{\small (law)}}{=} r^{d-2}\cp\Big(W_{[0,tr^{-2}]}^{0,1}\Big).
\end{equation}
The claim follows from equation \eqref{eq:capscaling} in combination with the first step.
\end{proof}

\subsubsection{Case $d=4$}
\label{S5.3.2}
The proof is based on methods presented in \cite[Chapter 10]{LL10}.
Fix $t>0$, let $B$ be the Brownian  motion driving $W_{[0,t]}^{0,1}$.

\begin{proof} 
We give the proof for the case $r=1$. A scaling argument as in \eqref{eq:capscaling} yields the general case.
First, note that
\be \label{eq:ineg.cp.Zt}
\cp\Big(\ws{0,1}{[0,t]}\Big) \lesssim \frac{t}{Z_t}, \qquad \mbox{where} \qquad Z_t=\inf_{y\in W_{[0,t]}^{0,1}}\int_0^t\, \int_{\cB(0,1)}\, \, G(y,B_u+z)\ \dd z\ \dd u.
\ee
Let us define $f(y) = \int_0^t \one\{y\in \cB(B_u,1)\}\dd u$ for $y\in\bbR^d$, and notice that $f(y) >0$ if and only if $y\in \ws{0,1}{[0,t]}$. Henceforth, we abbreviate $W=\ws{0,1}{[0,t]}$. By \eqref{eq:hittingprob}, we have one the one hand
\be
\int_W \int_W G(x,y) f(y) e_W(\dd x) \dd y = \int_W f(y) \dd y = c_\vol \ t,
\ee
and on the other hand,
\be
\int_W \int_W G(x,y) f(y) e_W(\dd x) \dd y \geq \int_W Z_t e_W(\dd x) = Z_t \cp(W),
\ee
from which we get \eqref{eq:ineg.cp.Zt}.
For a constant $c_0>0$ to be determined later,
\begin{equation}
\label{eq:capdiv}
\begin{aligned}
\E\bigg[\cp\Big(W_{[0,t]}^{0,1}\Big)^2\bigg]&=
\E\bigg[\one\{Z_t\leq c_0\log t\}\,\cp\Big(W_{[0,t]}^{0,1}\Big)^2\bigg]
+ \E\bigg[\one\{Z_t >c_0 \log t\}\,\cp\Big(W_{[0,t]}^{0,1}\Big)^2\bigg]\\
&\lesssim \E\bigg[\one\{Z_t\leq c_0\log t\}\,\cp\Big(W_{[0,t]}^{0,1}\Big)^2\bigg] + \Big( \frac{t}{c_0\log t} \Big)^2,\quad {\rm by} \eqref{eq:ineg.cp.Zt}.
\end{aligned}
\end{equation}
Note that by an application of the Cauchy-Schwarz inequality,
\begin{equation}
\label{eq:1CS}
\E\bigg[\one\{Z_t \leq c_0 \log t\}\,\cp\Big(W_{[0,t]}^{0,1}\Big)^2\bigg] \leq \P(Z_t\leq c_0\log t)^{1/2}\E\Big[\cp\Big(W_{[0,t]}^{0,1}\Big)^4\Big]^{1/2}.
\end{equation}
To estimate the right hand side in \eqref{eq:1CS} we use the a priori estimate
\begin{equation}
\label{eq:priori}
\E\Big[\cp\Big(W_{[0,t]}^{0,1}\Big)^4\Big]\leq c\ t^4,
\end{equation}
which may be proven as the corresponding second moment estimate in \eqref{eq:squaresum}
or via a scaling argument using Brownian scaling and the capacity scaling relation~\eqref{eq:scaling}.
Using Lemma \ref{lem:Z} below to handle the probability appearing on the right hand side of \eqref{eq:1CS} and choosing $c_0$ small enough such that $4-c/c_0 \leq 2$, we may conclude the proof.
\end{proof}

\begin{lemma}
\label{lem:Z}
There is $t_0>0$ such that for all $\gep$ small enough,
\begin{equation}
\label{eq:Zest}
\P(Z_t \leq \gep\log t)\leq t^{-c/\gep}, \qquad t\geq t_0.
\end{equation}
\end{lemma}

\begin{proof}[Proof of Lemma \ref{lem:Z}]

Let \be\label{eq:Gstar}
G^*(x,y) = \int_{z\in\cB(0,1)} G(x,y+z)\dd z.
\ee
and $G^*(x) = G^*(x,0)$.
We claim that there are $t_0>0$ and $c_0>0$ such that for all $\varepsilon$ small enough, for all $t\geq t_0$,
\be
\label{eq:claim}
\P\Big(\int_0^t G^*(B_u)\,\dd u\leq \varepsilon\log 2t\Big)\lesssim t^{-c_0/\varepsilon}.
\ee
We first show how one deduces Lemma~\ref{lem:Z} from this claim. To that end, we choose $\varepsilon$ small enough such that~\eqref{eq:claim} holds. Recall \eqref{eq:ineg.cp.Zt} and note that
\begin{equation}
Z_t = \inf_{y\in \ws{0,1}{[0,t]}} \int_0^t G^*(y,B_u) \dd u,
\end{equation}
and that there exists $C^*$ such that 
\be \label{eq:ratioGGstar}
\frac{1}{C^*} \leq \frac{G^*(x)}{G^*(y)} \leq C^*,\qquad ||x-y||\leq 1.
\ee
To show the existence of such a $C^*$ we use that $G^*(0)<\infty$ and that $G^*(x)\leq G^*(0)$, where both properties follow from the finiteness and monotonicity of $G$. 
We deduce that with $C^*= G^*(0)/\inf\{G^*(y):\, ||y||\leq 3\}$ we have the inequality
$G^*(x)\leq C^*G^*(y)$ for all $||x-y||\leq 1$ such that $\min\{||x||,||y||\}\leq 2$.
If $\min\{||x||,||y||\}>2$, then we note that for all $z\in\R^4$ with $||z||\leq 1$,
\begin{equation}
||x-z||\leq ||x-y||+ ||y-z||\leq 2||y+z||.
\end{equation}
Hence, after a possible increase of $C^*$ we may conclude the proof of~\eqref{eq:ratioGGstar}. Then,
\be
Z_t \geq \frac{Z'_t}{C^*},\quad \mbox{ with } \quad  Z'_t = \inf_{y\in B_{[0,t]}} \int_0^t G^*(y,B_u) \dd u.
\ee
To proceed, let $n\in\N$ such that
\be\label{eq:epsilon}
\varepsilon\in \Big(\frac{c_0\log t}{C^*nt}, \frac{c_0\log t}{C^*(n-1)t}\Big],\quad\mbox{ with } \frac{1}{0}=\infty.
\ee
Let $s>0$ and define
\be
Z'_{s,t,i} = \inf_{(i-1)/(2ns) \leq v \leq i/(2ns)} \int_0^t G^*(B_v,B_u) \dd u, \qquad 1\leq i \leq \lceil 2nst \rceil.
\ee
Then, for $i\leq \lceil 2nt^2 \rceil /2$,
\be\ba
\P(Z'_{t,t,i} \leq \gep \log t) &\leq \P\Big(\inf_{(i-1)/(2nt) \leq v \leq i/(2nt)} \int_{(i-1)/(2nt)}^{(i-1)/(2nt)+t/2} G^*(B_v,B_u) \dd u \leq \gep \log t\Big)\\
&= \P(Z'_{t,t/2,1} \leq \gep \log t),
\ea\ee
and with a similar argument for $i>\lceil 2nt^2 \rceil /2$,
\be
\P(Z'_t \leq \gep \log t) \leq \lceil 2nt^2 \rceil \P(Z'_{t,t/2,1} \leq \gep \log t).
\ee
Note that $\P(\sup_{0\leq v \leq 1/(2nt)} ||B_v|| >1) = \P(\sup_{0\leq v \leq 1} ||B_v|| > \sqrt{2nt}) \leq e^{-cnt}$, by Brownian scaling and Doob's inequality. Thus, using \eqref{eq:ratioGGstar},
\be \ba
\P(Z'_{t,t/2,1} \leq \gep \log t) &\leq e^{-nt} + \P\Big(Z'_{t,t/2,1} \leq \gep \log t, \sup_{0\leq s \leq 1/(2nt)} ||B_s|| \leq 1\Big)\\
&\leq e^{-nt} + \P\Big( \int_0^{t/2} G^*(B_u) \dd u \leq C^*  \gep \log t\Big)\\
&\lesssim t^{-c_0/(C^*\gep)},
\ea \ee
where the last estimate makes use of ~\eqref{eq:claim} and the relation~\eqref{eq:epsilon}.
Using~\eqref{eq:epsilon}, one may conclude the proof.
We are left with showing \eqref{eq:claim}.

Since $G^*$ is radial, harmonic outside $\cB(0,1)$, continuous on $\bar\cB(0,1)$ and $\lim_{||x||\nearrow\infty} G^*(x) =0$, 
\be \label{eq:decreaseGstar}
G^*(x)  = c ||x||^{-2},\qquad ||x||\geq 1,
\ee
see for instance Exercise 3.7 in \cite{MP10}.
Define the sequence of stopping times $\tau_0=0$ and $\tau_i = \inf\{ s\geq \tau_{i-1}\colon ||B_s|| \geq 2^i\}$, for $i\in\N$. From \eqref{eq:decreaseGstar}, we know that
\be
G^*(B_u) \geq \min\Big(c2^{-2i}, \inf_{x\in\cB(0,1)} G^*(x)\Big), \qquad \tau_{i-1} \leq u \leq \tau_i.
\ee
Moreover, combining~\eqref{eq:Gstar} and \eqref{eq:green} we see that
\be
G^*(B_u) \geq c2^{-2i}, \qquad \tau_{i-1} \leq u \leq \tau_i,\qquad i\in\N.
\ee
We obtain that
\be\label{eq:Gstarlowerbound}
\int_0^{\tau_N} G^*(B_u)\dd u \geq c\sum_{i=1}^N 2^{-2i} (\tau_i - \tau_{i-1}).
\ee
We now set $I_k = \one\{\tau_k - \tau_{k-1} < \gep 2^{2k}\}$, where $\gep\in(0,1)$. Using the strong Markov property and Brownian scaling we see that,
\be
\P(I_k=1) \leq \P\Big(\sup_{0\leq s \leq 4\gep} ||B_s|| > 1/2\Big) \leq e^{-c/\gep}, \qquad k\in\N.
\ee
Thus, the strong Markov property yields that the random variable $\sum_{1\leq k \leq N} I_k$ is stochastically dominated by a binomial random variable with parameters $N$ and $e^{-c/\gep}$. Therefore, well known tail estimates for the Binomial distribution show that 
\be \label{eq:inegIk}
\P\Big(\sum_{1\leq k \leq N} I_k \geq N/2\Big) \lesssim e^{-cN/\gep},
\ee
where $c$ and the proportionality constant are independent of $\varepsilon$.
Moreover, if $\sum_{1\leq k \leq N} I_k < N/2$, then as a consequence of~\eqref{eq:Gstarlowerbound}, $\int_0^{\tau_N} G^*(B_u) \geq \gep N/2$.
From this observation and \eqref{eq:inegIk}, we deduce that
\be \label{eq:ubintgstar}
\P\Big(\int_0^{\tau_N} G^*(B_u) \leq \gep N/2\Big) \lesssim e^{-cN/\gep}. 
\ee
It remains to replace $\tau_N$ in 
~\eqref{eq:ubintgstar} by $t$. For that we distinguish two cases.\\
\noindent {\bf (1)} $\varepsilon > (\log 2t)/t$. In this case we may write
\be
\P\Big(\int_0^t G^*(B_u) \dd u \leq \gep \log 2t\Big) \leq \P\Big(\int_0^{\tau_N} G^*(B_u) \dd u \leq \gep \log 2t\Big) + \P(\tau_N\geq t),
\ee
with $N = \lceil \frac14 \frac{\log 2t}{\log 2}\rceil$. Indeed, for the first term we use \eqref{eq:ubintgstar} with $\gep$ replaced by $(8\log2)\gep$ and for the second term, we have
\be
\P(\tau_N\geq t) 
 \leq e^{-c t^{1/2}},
\ee
for $t$ large enough, thanks to a standard small ball estimate, see \cite{LS01}. \\
\noindent {\bf (2)} $\varepsilon\leq (\log 2t)/t$. Note that if 
\begin{equation}
\label{eq:Gstarextreme}
\int_0^tG^*(B_u)\, \dd u\leq (\log t)^2/t,
\end{equation}
then there is $\bar u\in [0,t]$ such that $G^*(B_{\bar u})\leq (\log 2t)^2/t^2$. Thus, by the definition of $G^*$ and \eqref{eq:decreaseGstar} we see that $||B_{\bar u}||\geq t/\log 2t$.
Let $N = \lceil \frac14 \frac{\log 2t}{\log 2}\rceil$. We conclude that in particular the intersection of the event in~\eqref{eq:Gstarextreme} and $\{\tau_N\geq t\}$ is empty. We may now conclude in a similar fashion as in {\bf (1)}.

\end{proof}

\subsection{Large deviations estimates. Proofs of Lemmas \ref{lem:ldestimate} and \ref{ldestimate4d}}
\label{S5.4}
\begin{proof}[Proof of Lemma \ref{lem:ldestimate}]
We prove Lemma~\ref{lem:ldestimate} just for the case $r=1$. The extension to general $r$ can be done via a scaling argument as in Section~\ref{S5.3}. For simplicity, we assume that $t\in\bbN$, the extension to $t\in (0,\infty)\setminus\bbN$ being straightforward. Assume for the moment that
\be
\label{exp.mom.cap}
\E\big[ \exp\big(\gl \ \cp(\ws{0,1}{[0,1]}) \big)\big] < \infty, \qquad \gl>0.
\ee
Then, using~\eqref{eq:capseparation}, for $\gl >0$
\be
\ba
\P\Big(\cp\Big(\ws{0,1}{[0,t]}\Big) \geq j t\Big) &\leq \P\Big(\sum_{1\le i \le t} \cp\Big(\ws{0,1}{(i-1,i]}\Big) \geq j t\Big)\\
&\leq e^{-\gl j t} \E\bigg[ \exp\Big( \gl \sum_{1\le i \le t} \cp(\ws{0,1}{(i-1,i]})  \Big) \bigg]\\
&\leq \exp\Big\{ -\gl j t + t \log \E\Big(e^{\gl \cp(\ws{0,1}{[0,1]})}\Big) \Big\},
\ea
\ee
by the Markov property. By setting
\be
\gl > 2 \qquad \mbox{and} \qquad j_0 = \frac{1}{\gl - 2} \log\E\Big[ e^{\gl \cp(\ws{0,1}{[0,1]})} \Big],
\ee
we obtain the desired result. We now prove \eqref{exp.mom.cap}. The proof is inspired from Sznitman~\cite[Section 5]{S87}. Define the sequence of stopping times
\be 
T_0 = 0, \qquad T_{n+1} = \inf\{s\geq T_n \colon |B_{s+T_n} - B_{T_n}| \geq 1\},\qquad n\in\bbN_0.
\ee
If $N = \sup\{n\in\N_0\colon T_n \leq 1\}$, then $\ws{0,1}{[0,1]} \subseteq \bigcup_{k=1}^{N} \cB(B_{T_k}, 2)$, so $\cp(\ws{0,1}{[0,1]}) \leq cN$, by \eqref{eq:unionbound}. Therefore,
\be \label{eq:lemexp1}
\E\big[ \exp\big(\gl \ \cp(\ws{0,1}{[0,1]}) \big)\big] \leq \E\big[ \exp(c\gl N)\big].
\ee
By the Markov property, the increments $(T_i - T_{i-1})_{i\in\N}$ are iid. Thus, we may write
\be \label{eq:lemexp2}
\P(N\geq n) \leq \P(T_n \leq 1) = \P\Big(\sum_{i=1}^n (T_i - T_{i-1}) \leq 1\Big)\leq e^\gamma \E(e^{-\gamma T_1})^n,\quad \gamma>0.
\ee
From Sznitman~\cite[Lemma 5.1]{S87}, we may choose $\gamma$ large enough  such that $\E(e^{-\gamma T_1}) < e^{-c\gl}$, which, in combination with \eqref{eq:lemexp1} and \eqref{eq:lemexp2}, concludes the proof. 
\end{proof}

\begin{proof}[Proof of Lemma \ref{ldestimate4d}]
Recall \eqref{eq:ineg.cp.Zt}. Thus, there is $c_0>0$ such that 
\be
\P\Big( \cp(\ws{0,1}{[0,t]}) \geq j \frac{t}{\log t}\Big) \leq \P\Big(Z_t \leq \frac{c_0}{j} \log t\Big).
\ee
By Lemma \ref{lem:Z}, there exists $j_0\in(0,\infty)$ such that for $j\geq j_0$,
\be
\P\Big(Z_t \leq \frac{c_0}{j} \log t\Big) \leq t^{-cj},
\ee
from which we get
\be
\P\Big( \cp(\ws{0,1}{[0,t]}) \geq j \frac{t}{\log t}\Big) \leq t^{-cj}.
\ee
 We now generalise the estimate to arbitrary radius $r>0$. Since
\be
\cp(\ws{0,r}{[0,t]}) = r^2 \cp (\frac{1}{r} \ws{0,r}{[0,t]}) \stackrel{\mbox{(law)}}{=} r^2 \cp \ws{0,1}{[0,tr^{-2}]},
\ee
which is of the order
\be
r^2 \times \frac{tr^{-2}}{\log(tr^{-2})} = \frac{t}{\log t + 2 |\log r|} \sim \frac{t}{|\log r|}, \quad \mbox{since } t=\gep\sqrt{|\log r|}.
\ee
Therefore,
\be
\label{eq:capld4}
\P\Big(\cp(\ws{0,r}{[0,t]})\geq j\frac{t}{|\log r|}\Big) \leq t^{-cj}, \qquad j>j_0,
\ee
which holds after a possible increase of $j_0$.
\end{proof}

\end{document}